\documentclass[11pt]{article}

\usepackage{listings}
\usepackage{pdflscape}
\usepackage{epsfig}
\usepackage{lscape}
\usepackage{longtable}
\usepackage{amsmath,amssymb,amsthm}
\usepackage{mathabx}
\usepackage{graphicx}
\usepackage{color}
\usepackage{placeins}
\usepackage{url}
\usepackage{enumitem}
\def\sqr#1#2{{\vcenter{\vbox{\hrule height.#2pt
        \hbox{\vrule width.#2pt height#1pt \kern#2pt
        \vrule width.#2pt}
        \hrule height.#2pt}}}}
\def\approxleq{ \kern3pt \mbox{\raisebox{.6ex}{$<$}} \kern-8pt
  \mbox{\raisebox{-.6ex}{$\sim$}} \kern5pt}

\def\norm#1{\|#1 \|}

   \def\cT{{\cal T}}

\newlength{\len}
\newtheorem{theorem}{Theorem}[section]
\newtheorem{prop}{Proposition}[section]
\newtheorem{lemma}{Lemma}[section]
\newtheorem{corollary}{Corollary}[section]
\newtheorem{remark}{Remark}[section]
\newtheorem{assumption}{Assumption}[section]
\newtheorem{defn}{Definition}[section]
\newtheorem{example}{Example}

\def\R{\mathbb{R}} \def\E{\mathbb{E}}
\def\S{\mathbb{S}}

\def\X{\mathbb{X}}
\def\Y{\mathbb{Y}}
\def\Z{\mathbb{Z}}
\def\W{\mathbb{W}}

\begin{document}

\title{On the Asymptotic Superlinear Convergence of the Augmented Lagrangian Method for Semidefinite Programming with Multiple Solutions }

\author{Ying Cui\thanks {Department of Mathematics, National University of Singapore, 10 Lower Kent Ridge Road, Singapore ({\tt cuiying@nus.edu.sg}).}, \;Defeng Sun\thanks{Department  of  Mathematics, National University of Singapore, 10 Lower Kent Ridge Road, Singapore ({\tt matsundf@nus.edu.sg}). This research is supported in part by the Academic Research Fund under Grant R-146-000-207-112. } \  and \ Kim-Chuan Toh\thanks{Department of Mathematics, National University of Singapore, 10 Lower Kent Ridge Road, Singapore
({\tt mattohkc@nus.edu.sg}). This research is supported in part by the Ministry of Education, Singapore, Academic Research Fund under Grant R-146-000-194-112.
 }\\[10pt]
 }

\maketitle

\begin{abstract}
Solving
large scale convex semidefinite programming (SDP) problems has long been a challenging task numerically.
Fortunately, several powerful solvers including SDPNAL, SDPNAL+  and  QSDPNAL have recently been  developed to  solve
linear and convex quadratic SDP problems to high accuracy  successfully.
These solvers are based on the
augmented Lagrangian method (ALM) applied to the dual  problems with the subproblems being solved by  semismooth Newton-CG methods.
Noticeably, thanks to Rockafellar's general theory on the proximal point algorithms, the primal iteration sequence generated by the ALM enjoys an asymptotic
Q-superlinear convergence rate under a second order sufficient condition {for the primal problem}.
This second order sufficient
condition implies that the  primal problem has a
 unique solution, which can be restrictive in many applications.
For gaining more insightful interpretations on the high efficiency of these solvers,  in this paper
we conduct an asymptotic superlinear convergence analysis of the ALM for convex SDP when the  primal problem
has multiple solutions (can be unbounded).
Under a fairly mild second order growth condition,
we
prove that the primal iteration sequence generated by the ALM converges  asymptotically Q-superlinearly, while the dual feasibility and the dual objective function value converge  asymptotically R-superlinearly.
Moreover, by studying   the metric subregularity of the Karush-Kuhn-Tucker solution mapping,
we also provide sufficient conditions to guarantee the asymptotic R-superlinear convergence of the
dual iterate.
\end{abstract}

\medskip
{\small
\begin{center}
\parbox{0.95\hsize}{{\bf Keywords.}\; Semidefinite programming, augmented Lagrangian, second order growth condition, metric subregularity
}
\end{center}
}
\begin{center}
\parbox{0.95\hsize}{{\bf AMS subject classifications:}\; 90C25, 90C33, 65K05}
\end{center}

\section{Introduction}
Let $\S^n$ be the space  of 
$n\times n$ real symmetric matrices equipped with {the}
standard trace inner product $\langle \cdot,\cdot\rangle$ and its induced Frobenius
norm $\|\cdot\|$.
{We use} $\S_+^n$ to denote the cone of  $n\times n$ symmetric positive semidefinite matrices in $\S^n$. We write $X\succeq 0$ if $X\in\S^n_+$  and $X\succ 0$ if $X$ is symmetric positive definite.

Semidefinite programming (SDP) is an extremely
important and active research area in modern optimization. 
Among various SDP models, the most fundamental one is the following  standard {primal} linear SDP:
\begin{equation}\label{intro:linearsdp}
\min \bigg\{\langle C,X\rangle\; |\; \mathcal{A}X = b, \; X\in\S_+^n\bigg\},
\end{equation}
where $\mathcal{A}:\S^n\to\R^m$ is a linear {map}, 
$C\in \S^n$ and $b\in\R^m$ are given data.
 The dual of (\ref{intro:linearsdp}) is given by
\begin{equation}\label{intro:dualsdp1}
\begin{array}{ll}
\max \bigg\{ \langle b,y\rangle\;|\; \mathcal{A}^*y + S = C,\; S\in\S_+^n\bigg\},
\end{array}
\end{equation}
where $\mathcal{A}^*:\R^m\to\S^n$ is the adjoint {map} 
of $\mathcal{A}$.

The problem (\ref{intro:linearsdp}) arises frequently from the SDP relaxations of numerous NP-hard
  combinatorial optimization problems, such as frequency assignment problems~\cite{eisenblatter2002frequency},
   maximum stable set problems~\cite{grotschel1986relaxations}, quadratic assignment and binary integer quadratic problems~\cite{lovasz1991cones}, etc.
As a consequence,
much 
{effort has been put into designing algorithms} 
for solving large scale semidefinite programming (SDP) efficiently.
{It is widely recognized that
interior point methods (IPMs) such as those implemented in
\cite{Sturm99, TTT99,TTT03} are highly successful in solving small and medium sized SDPs;
see~\cite{Todd01} for a nice survey on this topic. However, IPMs are generally inefficient for solving large scale
SDPs due to their inherent poor computational scalability and expensive memory requirement. To overcome
these drawbacks, various attempts on using first order methods to solve special classes of large SDPs
have been made in recent years. These include the boundary point method~\cite{malick2009regularization},
a directly extended alternating direction method of multipliers (ADMM) \cite{wen2010alternating}, a two-easy-block-decomposition hybrid proximal extragradient method~\cite{monteiro2014first}, and
a convergent multi-block ADMM$+$ \cite{STY15}.
}
{While the first order methods just mentioned are reasonably efficient in solving some large scale SDPs, they
may become inefficient when higher accuracy solutions are required and more seriously, they can fail badly
when solving more difficult problems as demonstrated in \cite{yang2015sdpnal+}.

In contrast,
the solver SDPNAL~\cite{zhao2010newton}
developed by Zhao, Sun and Toh, which made use of second-order information,
is much more efficient in solving large SDPs to high accuracy.
This powerful solver, designed for large scale linear SDP problems of the form (\ref{intro:linearsdp}),
is based on an augmented Lagrangian method (ALM) applied to the dual problem (\ref{intro:dualsdp1})
wherein the subproblems are approximately solved by the  semismooth Newton-CG method.
Extensive numerical experiments have shown that it
is highly efficient for solving large scale SDPs with non-degenerate primal optimal solutions.
}

A more complicated  linear SDP problem is the following so-called doubly nonnegative SDP:
 \begin{equation}\label{intro:dsdp}
\min \bigg\{\langle C,X\rangle\; |\; \mathcal{A}X = b, \; X\in\S_+^n,\; X\geq 0\bigg\}.
\end{equation}
Even though it can be reformulated as a standard SDP by introducing  additional constraints $X'= X$ and $X'\geq 0$, the 
{reformulated} 
problem is  
{usually}
 degenerate and thus SDPNAL may fail to solve it efficiently. To overcome this difficulty,
  an enhanced version of SDPNAL, called SDPNAL+, was developed by Yang, Sun and Toh ~\cite{yang2015sdpnal+} recently. With a majorized semismooth Newton-CG method for {solving}
the inner subproblems in the ALM,
the new solver {can successfully compute} solutions of high accuracy  for large scale doubly nonnegative SDPs.

The dual based ALM {coupled with a semismooth Newton-CG algorithm has also} been extended to other classes of SDP problems. Jiang, Sun and Toh~\cite{jiang2013solving} have employed this approach to solve
the least squares SDP problem:
\begin{equation}\label{intro:ls}
\min \bigg\{ \frac{1}{2}\|\mathcal{F}X  - d\|^2 + \langle C,X\rangle\;|\; \mathcal{A}X = b, \; X\in\S_+^n\bigg\},
\end{equation}
where $\mathcal{F}:\mathbb{S}^n\to\mathbb{R}^{m'}$ is a linear map and $d\in\mathbb{R}^{m'}$ is a given vector.
Most recently, this idea is adopted by Li, Sun and Toh~\cite{li2015qsdpnal} for developing the solver QSDPNAL to
 deal with the following convex quadratic SDP (QSDP) with a  given self-adjoint positive semidefinite operator  $\mathcal{P}:\S^n\to\S^n$:
\begin{equation}\label{intro:qsdp}
\min \bigg\{\frac{1}{2}\langle X, \mathcal{P}X\rangle +\langle C,X\rangle\;|\; \mathcal{A}X = b, \; X\in\S_+^n\bigg\}.
\end{equation}

 It is well known that the ALM applied to the dual problem 
is equivalent to a proximal point algorithm (PPA) 
{applied}
to its primal form~\cite{rockafellar1976augmented}.
So for all of the  ALM based solvers mentioned above  for solving SDP problems,
the primal iteration is proven to converge asymptotically
superlinearly to an optimal solution
under a second order sufficient condition via~\cite[Proposition 3 and Theorem 2]{rockafellar1976augmented}.
However, this second order sufficient condition can be restrictive because it  fails  to hold when 
{the primal}
SDP problem
has multiple solutions. For better understanding {of} 
the ALM for solving  SDPs, in this paper we aim to remove this restriction by
 conducting extensive  analysis on both the second order variational properties of the positive semidefinite cone and 
  the metric subregularity of {the solution mappings of linearly constrained convex SDPs}. In particular, 
   assuming {that the  problem admits} a Karush-Kuhn-Tucker (KKT) point with a
   partial strict complementarity property,  we
   prove that the sequence $\{X^k\}$ generated by the ALM converges  asymptotically Q-superlinearly, while the dual feasibility and the dual objective function value converge  asymptotically R-superlinearly. We also study sufficient conditions for ensuring the metric subregularity of the KKT solution mapping, which is shown to guarantee the asymptotic R-superlinear convergence of the dual iteration sequence.
 

The remaining parts of this paper are organized as follows. In the next section, we introduce some definitions and preliminary results on variational analysis and maximal monotone operators. In Section 3, we conduct extensive studies on sufficient conditions for the metric subregularity of {the solution mappings of linearly constrained convex SDPs}. Section 4 is devoted to  the asymptotic superlinear
convergence analysis of the
ALM for solving SDP problems. Under the existence of a strictly feasible solution, {we also design}
new easy-to-implement stopping criteria {for} the ALM 
in this section.
We conclude our paper and make some comments in the final section.

Below we list other symbols and notation to be used in our paper.
\begin{itemize}
\item Let $\R^{m\times n}$ be the linear space of $m \times  n$ real matrices equipped with the inner product $\langle X,Y\rangle = \textup{tr}(X^TY)$ for any $X,Y\in\R^{m\times n}$. Here $\textup{tr}(\cdot)$ denotes the trace, i.e., the sum of all the diagonal entries,  of a {square} matrix. Let $\mathcal{O}^n$ be the set of $n\times n$ orthogonal matrices.  We also  use $0_n$ and $I_n$ to denote the $n\times n$ zero matrix and identity matrix, respectively. For any $X\in\mathbb{S}^n$, $\lambda_{\textup{max}}(X)$ and $\lambda_{\textup{min}}(X)$ represent  the largest and the smallest eigenvalues of $X$, respectively.
\item We use $\X$, $\Y$, $\Z$ and $\W$  to denote some finite dimensional real Euclidean spaces.
For any convex function $p:\mathbb{X}\to(-\infty, +\infty]$, we denote its effective domain as $\textup{dom}(p) := \{x\in\mathbb{X}\mid p(x) < \infty\}$ and its conjugate  as $p^*(u)  := \sup_{x\in\mathbb{X}}\{\langle x, u\rangle  - p(x)\}$, $u\in\mathbb{X}$. For any  $x\in\mathbb{X}$ and  $\rho>0$, {we define} $\mathcal{B}_{\mathbb{X}}(x, \rho):=\{y\in\mathbb{X}\mid \|y-x\|\leq \rho\}$.
For any linear map $\mathcal{A}:\mathbb{X}\to\mathbb{Y}$, we use  $\textup{Range}(\mathcal{A})$ to denote the range space of $\mathcal{A}$.
\item  Let $\alpha \subseteq \{1,...,m\}$ and $\beta \subseteq \{1,...,n\}$ be two index sets. For any $Z \in \R^{m\times n}$, we write $Z_{\alpha\beta}$ {to be} the $|\alpha|\times |\beta|$ sub-matrix of $Z$ obtained by
removing all the
rows of $Z$  not in $\alpha$ and all the columns of $Z$  not in $\beta$. Denote $\textup{diag}(x_i \mid i\in \alpha)$ as the $|\alpha|\times|\alpha|$ diagonal matrix whose $i$-th diagonal entry is the $i$-th component of $x_{\alpha}$, $i=1,\ldots, |\alpha|$.
\item Let $D\subseteq \X$ be a set.  For any $x\in \X$,
define $\textup{dist}(x,D):= \inf_{d\in D} \|x-d\|$. {We let}  $\delta_{D}(\cdot)$ to 
be the indicator function over $D$, i.e., {$\delta_{D}(x) = 0$ if $x\in D$, and $\delta_D(x) = \infty$ if $x\not\in D$.}
\item
If $D\subseteq \X$ is  a convex set, 
we  use $\textup{ri}(D)$ to denote its relative interior.
{For a given closed convex set
$D\subseteq \X$}, the metric projection of $x\in\X$ onto $D$ is defined by $\Pi_{D}(x) := \arg\min \{\|x-d\|\mid d\in D\}$. For any $x\in D$,
{we use $\mathcal{T}_{D}(x)$ and $\mathcal{N}_{D}(x)$ to denote the tangent and normal cone of $D$ at $x$, 
respectively as in standard convex analysis~\cite{rockafellar2015convex}. }
 If $D$ is a closed convex cone, we use $D^\circ$ and $D^*$ to denote the polar of $D$ and  the dual of $D$, respectively, i.e.,
 $D^\circ:=\{x\in \mathbb{X}\mid \langle x, d\rangle \leq 0,\; \forall\,d\in D\}$ and $D^*:=-D^\circ$. 

\end{itemize}

\section{Preliminaries}

Let $F:\X\rightrightarrows\Y$ {be}
a multi-valued mapping. The graph of the mapping $F$ is defined as
$\text{gph}(F) := \{(x,y)\in\X\times \Y\,\mid\,y\in F(x)\}$.
The following definition  of metric subregularity  is taken from~\cite[Section 3.8(3H)]{dontchev2009implicit}.

\begin{defn}\label{defn:metric11}
A multi-valued mapping  $F:\X\rightrightarrows\Y$ is
said to be
metrically subregular at $\bar x\in\X$ for $\bar y\in\Y$ with modulus $\kappa\geq 0$ if $(\bar x, \bar y)\in \textup{gph}(F)$ and   there exist  neighborhoods $\mathcal{U}$ of $\bar x$ and  $\mathcal{V}$ of $\bar y$ such that
\begin{equation}\label{defn:metric}
\textup{dist}(x, F^{-1}(\bar y))\leq \kappa\,\textup{dist}(\bar y, F(x)\cap \mathcal{V}), \;\forall\, x\in\mathcal{U},
\end{equation}
or equivalently, $F$ is said to be metrically subregular at $\bar x$ for $\bar y$ with modulus $\kappa\geq 0$
if there exists a neighborhood $\mathcal{U}'$ of $\bar{x}$ such that
\begin{equation}\label{equiv:metricsub}
\textup{dist}(x, F^{-1}(\bar y))\leq \kappa\,\textup{dist}(\bar y, F(x)), \;\forall\, x\in\mathcal{U}'.
\end{equation}

\end{defn}
The next result, which provides a convenient way to check the metric  subregularity of  the subdifferential of a proper closed convex function,  is proven 
in~\cite[Theorem 3.3]{aragon2008characterization}.

\begin{prop}\label{pre:metricregular:thm}
Let $\mathcal{H}$ be a real Hilbert space endowed with the inner product $\langle \cdot,\cdot\rangle$ and $\theta:\mathcal{H}\to(-\infty, +\infty]$ be a proper lower semicontinuous convex function. Let $\bar{v}, \bar{x}\in\mathcal{H}$ satisfy $(\bar{x},\bar{v})\in\textup{gph}(\partial p)$. Then $\partial \theta$ is metrically subregular at $\bar{x}$ for $\bar{v}$ if and only if there exist a neighborhood $\mathcal{U}$ of $\bar{x}$ and a  constant $\kappa>0$ such that
\begin{equation}\label{pre:metricregular:ineq}
 \theta(x)\geq \theta(\bar{x}) + \langle \bar{v}, x-\bar{x}\rangle  + \kappa\,\textup{dist}^2(x, (\partial \theta)^{-1}(\bar{v})), \;\forall\, x\in\mathcal{U}.
\end{equation}
\end{prop}

A multi-valued mapping ${F}:\X\rightrightarrows\Y$ is said to be  polyhedral if its graph is the union of finitely many polyhedral convex sets.
Below is a fundamental result from Robinson~\cite{robinson1981some} on multi-valued polyhedral mappings.
\begin{prop}\label{prop:polyhedral}
Let ${F}:\X\rightrightarrows\Y$ be a multi-valued polyhedral mapping and $(\bar{x},\bar{y})\in\textup{gph}(F)$. Then $F$
is locally upper Lipschitz continuous at $\bar{x}$, i.e., there exist a constant $\kappa>0$ and a neighborhood $\mathcal{U}$ of $\bar{x}$ such that
$$
F(x)\subseteq F(\bar x) + \kappa\|x-\bar x\|\mathcal{B}_{\mathbb{Y}}(0,1), \;\forall\, x\in\mathcal{U}.
$$
\end{prop}


In our subsequent discussions, we also need the concept of bounded linear regularity of a collection of closed convex sets, which can be found from, e.g., ~\cite[Definition 5.6]{bauschke1996projection}.
\begin{defn}
Let $D_1, D_2, \ldots, D_m\subseteq\X$ be closed convex sets for some positive integer $m$.
{Suppose that
 $D := D_1\cap D_2 \cap\ldots\cap D_m$ is non-empty.}
 The collection $\{D_1, D_2, \ldots, D_m\}$ is said to be boundedly linearly regular if for every bounded set $\mathbb{B}\subseteq \X$, there exists a constant $\kappa >0$ such that
$$
\textup{dist}(x, D) \leq \kappa \max\left\{\textup{dist}(x, D_1), \ldots, \textup{dist}(x, D_m)\right\}, \;\forall\, x\in\mathbb{B}.
$$
\end{defn}
A sufficient condition to guarantee the property of bounded linear regularity is established  in~\cite[Corollary 3]{bauschke1999strong}.
\begin{prop}\label{prop:boundedlinear}
Let $D_1, D_2, \ldots, D_m\subseteq\X$ be closed convex sets for some positive integer $m$. Suppose that $D_1, D_2, \ldots, D_r$ are polyhedral for some $r\in\{0,1,\ldots, m\}$. Then a sufficient condition for
 $\{D_1, D_2, \ldots, D_m\}$ to be boundedly linearly regular is
 $$
 \bigcap_{i=1,2, \ldots, r} D_i \quad \cap\;  \bigcap_{i = r+1, \ldots, m} \textup{ri}\,(D_i)\neq \emptyset.
 $$
\end{prop}

In the following, we shall present an equivalent result on the metric subregularity of maximal monotone operators.
Consider the inclusion problem:
$$
0\in\Gamma(x) + \mathcal{T}(x),\; x\in\X,
$$
 where $\Gamma:\X\to\X$ is a continuous monotone operator and $\mathcal{T}:\X\rightrightarrows\X$ is a maximal monotone operator.
Define {the} mapping $\mathcal{R}:\X\to\X$ {by}
$$\mathcal{R}: = \mathcal{I} - (\mathcal{I}+\mathcal{T})^{-1}(\mathcal{I}- \Gamma),
$$
where $\mathcal{I}:\X\to\X$ is the identity operator.
Then one can see from~\cite{Minty1962monoton} that
$$
x\in(\Gamma + \mathcal{T})^{-1}(0) \;\Longleftrightarrow \; x\in\mathcal{R}^{-1}(0).
$$
The following proposition is on the equivalence between the metric subregularity of $\mathcal{R}$ and $\Gamma + \mathcal{T}$ at any $\bar{x}\in\mathcal{R}^{-1}(0)$ for the origin.
\begin{prop}\label{prop:equivmetric}
Suppose that $\mathcal{R}^{-1}(0)\neq \emptyset$. Let $\bar{x}\in\mathcal{R}^{-1}(0)$. Consider the following two statements:
\\[5pt]
(a) $\mathcal{R}$ is metrically subregular at $\bar{x}$ for the origin with modulus $\kappa_1\geq 0$ along with a neighborhood $\mathcal{B}_{\mathbb{X}}(\bar{x}, \rho_1)$, i.e.,
\begin{equation}\label{ebdefn:ineq}
\textup{dist}\;(x, \mathcal{R}^{-1}(0))\leq \kappa_1\|\mathcal{R}(x)\|, \; \forall\, x\in\mathcal{B}_{\mathbb{X}}(\bar{x}, \rho_1);
\end{equation}
(b) $\Gamma + \mathcal{T}$ is metrically subregular at $\bar{x}$ for the origin with modulus $\kappa_2\geq0$ along with a neighborhood $\mathcal{B}_{\mathbb{X}}(\bar{x}, \rho_2)$, i.e.,
\begin{equation}\label{metricsub:ineq}
\textup{dist}(x,(\Gamma + \mathcal{T})^{-1}(0))\leq \kappa_2\;\textup{dist}(0,(\Gamma+ \mathcal{T})(x)), \; \forall \, x\in\mathcal{B}_{\mathbb{X}}(\bar{x}, \rho_2).
\end{equation}
Then, the inequality (\ref{ebdefn:ineq}) implies the inequality (\ref{metricsub:ineq}) with $\rho_2 = \rho_1$ and $\kappa_2 = \kappa_1$.
Conversely,  if  the inequality (\ref{metricsub:ineq}) holds and there exists $\tau\geq 0$ such that
 $\Gamma$ is  Lipschitz continuous  on $\mathcal{B}_{\mathbb{X}}(\bar{x}, (1+\tau)^{-1}\rho_2)$ with modulus $\tau$,
then the inequality (\ref{ebdefn:ineq}) holds with {$\rho_1 = (1+\tau)^{-1}\rho_2$} and $\kappa_1 =  1 + (1+\tau)\kappa_2$.
\end{prop}
\begin{proof}
In~\cite[Theorem 3.1]{dong2009extension}, Dong proved the equivalence of parts (a) and (b),
 with $\mathcal{B}_{\mathbb{X}}(\bar{x}, \rho_1)$ in (\ref{ebdefn:ineq}) being replaced by   $\{x\in\X \mid \textup{dist}(x,\mathcal{R}^{-1}(0))\leq  \epsilon_1\}$  for some $\epsilon_1>0$ and $\mathcal{B}_{\mathbb{X}}(\bar{x}, \rho_2)$ in (\ref{metricsub:ineq}) being replaced by  $\{x\in\X \mid \textup{dist}(x,(\Gamma + \mathcal{T})^{-1}(0))\leq  \epsilon_2\}$ for some $\epsilon_2>0$, respectively. The proof of Proposition \ref{prop:equivmetric} can be conducted in a similar way as {in}
 \cite[Theorem 3.1]{dong2009extension}. For brevity,  we omit the details here.
\end{proof}

\section{The metric subregularity of solution mappings}

In this section, we shall discuss the metric subregularity of the solution mappings for solving 
linearly constrained convex semidefinite programming with multiple solutions. These problems can be cast into the following form:
\begin{equation}\label{eb:opt}
\begin{array}{cl}
\min &  \theta(x):=  h(\mathcal{F}x) +\langle c,x\rangle + p(x) \\[5pt]
\text{s.t.} &  b- \mathcal{A}x \in \mathcal{Q}^\circ,
\end{array}
\end{equation}
where $\mathcal{F}:\X\to\W$ and $\mathcal{A}:\X\to\Y$  are {given linear maps, $\mathcal{Q}\subseteq\mathbb{Y}$ is a given convex polyhedral cone,
$c\in \X$ and $b\in\Y$ are given data},
$p:\X\to (-\infty, +\infty]$ is a closed proper convex  function,
$h:\W\to (\infty, +\infty]$ is continuously differentiable on  $\text{dom}(h)$, which is assumed to be a non-empty  open convex set,
  and is also strictly convex on any convex subset of $\textup{dom}(h)$.
The dual of problem (\ref{eb:opt}) can be  written, in
 its equivalent minimization form,  as
\begin{equation}\label{dual2}
\begin{array}{cl}
\min &  \delta_{\mathcal{Q}}(y) -\langle b,y\rangle  + h^*(-w)  + p^*(-s)\\[5pt]
\text{s.t.} & \mathcal{A}^*y  + \mathcal{F}^*w + s = c.
\end{array}
\end{equation}
For notational convenience, define $\Z:=\Y \times \W\times \X$ and for any $(y,w,s)\in\Y\times \W\times \X$,  write $z:=(y,w,s)$.

The Lagrangian function $l$
 associated with {problem (\ref{dual2}) is given by}
\begin{equation}\label{defn:l}
l(z,x):= \delta_{\mathcal{Q}}(y)
-\langle b,y\rangle  + h^*(-w)  + p^*(-s) + \langle x, \mathcal{A}^*y  + \mathcal{F}^*w + s -c\rangle,\;\forall\,(z,x)\in\Z\times \X.
\end{equation}
Define {the}  functions $\psi$  and $\phi$ {by}
\begin{equation}\label{defn:phipsi}
\psi(z) := \sup_{x\in\X} l(z,x),\; \forall\, z\in\Z,\quad  \quad
\phi(x) := \inf_{z\in\Z} l(z,x),\; \forall\, x\in\X.
\end{equation}
Moreover, we define {the} mapping $\mathcal{T}_l:\Z\times \X\rightrightarrows\Z\times \X$ as
\begin{equation}\label{defn:Tl}
\mathcal{T}_l(z,x):=\big\{(u,v)\in\Z\times \X\; {\mid} \; (u,-v)\in\partial l(z,x)\big\}, \;\forall \;(z,x)\in\Z\times \X
\end{equation}
and {the} mappings  $\mathcal{T}_\psi:\Z\rightrightarrows\Z$ and $\mathcal{T}_\phi:\X\rightrightarrows\X$ as
\begin{equation}\label{defn:Tphi}
\mathcal{T}_\psi(z):=\partial \psi(z), \;\forall\, z\in\Z,\quad \quad
\mathcal{T}_\phi(x):=- \partial \phi(x), \;\forall\, x\in\X.
\end{equation}

Assume that problem (\ref{dual2}) admits at least one optimal solution $(\bar{y}, \bar{w}, \bar{s})\in\Y\times\W\times \X$.  Let $\mathcal{M}_\psi(\bar{z})\subseteq\X$ denote the set of Lagrangian multipliers corresponding to $\bar{z}$, i.e.,
$\bar{x}\in\mathcal{M}_\psi(\bar{z})$ if and only if $(\bar{y},\bar{w}, \bar{s}, \bar{x})$ solves the following KKT system:
 \begin{equation}\label{kkt}
\left\{\begin{array}{ll}
0\in -b + \mathcal{A}{x} +\mathcal{N}_{\mathcal{Q}}(y),\;
0\in \mathcal{F}x - \partial h^*(-w),\;
0\in x -\partial p^*(-s),\\[5pt]
0 = c - (\mathcal{A}^*y + \mathcal{F}^*w + s),
\end{array}\right. \quad (y,w,s,x)\in\Z\times \X.
\end{equation}
It can be easily checked  that if  $(\bar{z}, \bar{x}) = (\bar{y}, \bar{w}, \bar{s}, \bar{x})\in\Y\times \W\times \X\times \X$ is a solution to the KKT system (\ref{kkt}), then $(\bar{x}, \bar{y})$ solves the following KKT system:
 \begin{equation}\label{kktp}
\left\{\begin{array}{ll}
0 \in c - \mathcal{A}^*y+\mathcal{F}^*\nabla h(\mathcal{F}x)  + \partial p(x),\\[5pt]
0\in \mathcal{A}{x} - b +\mathcal{N}_{\mathcal{Q}}(y),
\end{array}\right. \quad (x,y)\in\X\times \Y;
\end{equation}
conversely, if $(\bar{x}, \bar{y})\in\X\times \Y$ solves (\ref{kktp}), then for $\bar{w} = -\nabla h(\mathcal{F}\bar{x})$ and  $\bar{s} = c-\mathcal{A}^*\bar{y} - \mathcal{F}^*\bar{w}$, $(\bar{z}, \bar{x}) = (\bar{y}, \bar{w}, \bar{s}, \bar{x})\in\Y\times \W\times \X\times \X$ solves the KKT system (\ref{kkt}). Let $\mathcal{M}_\phi(\bar{x})\subseteq\Z$ be the set of Lagrangian multipliers corresponding to $\bar{x}\in\mathcal{M}_\psi(\bar{z})$.

\subsection{The metric subregularity of $\mathcal{T}_\phi$}

Assume that the KKT system (\ref{kkt}) or (\ref{kktp}) admits at least one solution.
%
%
%
%
%
It is known from~\cite[Theorem 30.4 \text{and} Corollary 30.5.1]{rockafellar2015convex} that $(\bar{z}, \bar{x})\in\Z\times \X$ solves the KKT system (\ref{kkt}) if and only if
$\bar{z}\in\Z$ solves problem (\ref{dual2})  and  $\bar{x}\in\X$ solves problem (\ref{eb:opt}).
To further characterize  $\mathcal{T}^{-1}_{\phi}(0)$, we need the following invariant
%
%
%
 property of $\mathcal{F}x$ over  $x\in\mathcal{T}^{-1}_{\phi}(0)$, whose proof readily follows from the well-known existing techniques in the literature~\cite{mangasarian1988simple, luo1992linear,tseng2010approximation}.

\begin{lemma}\label{lemma:invariant}
The value $\mathcal{F}{x}$ is invariant over $x\in\mathcal{T}^{-1}_{\phi}(0)$, i.e.,
for any $x', x''\in\mathcal{T}^{-1}_{\phi}(0)$, we have $\mathcal{F}{x'} = \mathcal{F}x''$.
\end{lemma}
Take an arbitrary point $\bar{x}\in\mathcal{T}^{-1}_{\phi}(0)$ and denote
\begin{equation}\label{defn:ubar}
\begin{array}{ll}
\bar{\zeta}: = \mathcal{F}\bar{x}, \quad
 \bar{\eta}: = \mathcal{F}^*\nabla h(\bar{\zeta}) +c,\quad
 \overline{\mathcal{V}}:=\{x\in\X\; {\mid}\;  \mathcal{F}x = \bar{\zeta}\}.  \end{array}
 \end{equation}
We define two multi-valued mappings $\mathcal{G}_1:\Y\rightrightarrows\X$ and $\mathcal{G}_2:\Y\rightrightarrows\X$  by
 \begin{equation}\label{defn:mappingG}
 \mathcal{G}_1(y) := (\partial p)^{-1}(\mathcal{A}^*y - \bar{\eta}),\quad  \quad
\mathcal{G}_2(y) := \{x \;{\mid}\; 0\in \mathcal{A}x -b + \mathcal{N}_{\mathcal{Q}}(y)\}, \; \forall\, y\in\Y.
 \end{equation}
Then, from  (\ref{eb:opt}), (\ref{kkt}), Lemma \ref{lemma:invariant} and the arguments above Lemma \ref{lemma:invariant},   we immediately obtain the following useful observation  {for} the optimal solution set $\mathcal{T}^{-1}_{\phi}(0)$.

 \begin{prop}\label{prop:Omega}
 Assume that $(\bar{y},\bar{w}, \bar{s}, \bar{x})\in\Z\times \X$ solves the KKT system (\ref{kkt}).
Then  the  optimal solution set $\mathcal{T}^{-1}_{\phi}(0)$ to problem (\ref{eb:opt}) can be characterized as
%
$$\begin{array}{ll}
\mathcal{T}^{-1}_{\phi}(0)  = \{x\in\X \; {\mid} \; \mathcal{F}x = \bar{\zeta}, \; 0\in \bar{\eta} + \partial p({x}) - \mathcal{A}^*\bar{y},\; 0\in \mathcal{A}x - b+\mathcal{N}_{\mathcal{Q}}(\bar{y})\} = \overline{\mathcal{V}}\cap \mathcal{G}_1(\bar{y})\cap \mathcal{G}_2(\bar{y}).
\end{array}
$$\end{prop}

{To analyse the metric subregularity of $\cT_\phi$, we will need the following assumption later.}

\begin{assumption}\label{ass:f}
The following local growth conditions hold:
 \\
(i) For any ${w}\in\textup{dom}(h)$, there exist $\kappa_1>0$  and a neighborhood $\mathcal{W}\subseteq\W$ of $w$ such that
$$
h(w')\geq h({w}) + \langle \nabla h({w}), w' - {w}\rangle  + \kappa_1\|w'  -{w}\|^2, \;\forall\, w'\in\mathcal{W}.
$$
(ii) For any $(x,v)\in\textup{gph}(\partial p)$, there exist $\kappa_2>0$  and a neighborhood $\mathcal{U}\subseteq\X$ of $x$ such that
$$
p(x')\geq p(x) + \langle v, x' - {x}\rangle  + \kappa_2\;\textup{dist}^2(x', (\partial p)^{-1}(v)), \;\forall \, x'\in\mathcal{U}.
$$
\end{assumption}


We say {that} for problem (\ref{eb:opt}),
 the second order growth condition holds at an optimal solution $\bar{x}\in \mathcal{T}^{-1}_{\phi}(0)$ with respect to the set $ \mathcal{T}^{-1}_{\phi}(0)$  if  there exist $\kappa>0$ and a neighborhood $\mathcal{U}$ of $\bar{x}$ such that
\begin{equation}\label{defn:quadraticGrowth}
\theta(x) \geq \theta(\bar{x}) + \kappa\,\textup{dist}^2(x,  \mathcal{T}^{-1}_{\phi}(0)), \;\forall\, x\in\mathcal{U}\cap\{x\in\X\,\mid\, b - \mathcal{A}x \in\mathcal{Q}^\circ\}.
\end{equation}
Consider an arbitrarily fixed point $\bar{x}\in\mathcal{T}_{\phi}^{-1}(0)$. It can be seen from Proposition \ref{pre:metricregular:thm} that
the operator $\mathcal{T}_\phi$ is metrically subregular at $\bar{x}$ for the origin if and only if the second order growth condition (\ref{defn:quadraticGrowth}) holds at $\bar{x}$ with respect to $\mathcal{T}^{-1}_{\phi}(0)$.
Thus, we can study the metric subregularity of $\mathcal{T}_{\phi}$ at $\bar{x}$ for the origin via the second order growth condition (\ref{defn:quadraticGrowth}).
The following lemma is convenient  for our later discussions. \begin{lemma}\label{lemma:polyhdist}
Let $\bar{x}\in\mathcal{T}_\phi^{-1}(0)$ and
$\bar{y}\in\mathcal{M}_{\phi}(\bar{x})$. Then there exist a constant $\kappa>0$ and a neighborhood  $\mathcal{U}$ of $\bar{x}$ such that
$$\textup{dist}(x, \mathcal{G}_2(\bar{y}))\leq \kappa \; \textup{dist}(b-\mathcal{A}x, \mathcal{N}_{\mathcal{Q}}(\bar{y})), \;\forall\, x\in\mathcal{U}.
$$\end{lemma}
\begin{proof}
Define {the} subspace $\Xi_1\subseteq\X\times \Y$ and {the} polyhedral set $\Xi_2\subseteq\X\times \Y$
{by}
$$ \Xi_1 = \{(x,q)\in\X\times \Y\,\mid\, b - \mathcal{A}x = q\},\quad\quad \Xi_2 = \{(x,q)\in\X\times \Y\,\mid\, q\in\mathcal{N}_{\mathcal{Q}}(\bar{y})\}.$$
Denote
 $\widetilde{\mathcal{G}}_2:=\Xi_1\cap\Xi_2$, which is non-empty as $(\bar{x}, b-\mathcal{A}\bar{x})\in\widetilde{\mathcal{G}}_2$.
Since $\Xi_1$ and $\Xi_2$ are polyhedral sets, we know from Proposition \ref{prop:boundedlinear}
 that the collection $\{\Xi_1, \Xi_2\}$ is boundedly linearly regular.  Therefore, there exist a constant $\kappa>0$ and a neighborhood $\mathcal{U}$ of $\bar{x}$ such that
$$
\text{dist}((x,b-\mathcal{A}x), \widetilde{\mathcal{G}}_2)\leq
\kappa\big(\text{dist}((x,b-\mathcal{A}x),\Xi_1) +
\text{dist}((x,b-\mathcal{A}x),\Xi_2)\big) =  \kappa\;\text{dist}(b-\mathcal{A}x,\mathcal{N}_{\mathcal{Q}}(\bar{y})).
$$
Thus, by noting that there exists $(x',w')\in\widetilde{\mathcal{G}}_2$ such that
$$\text{dist}((x,b-\mathcal{A}x), \widetilde{\mathcal{G}}_2) = \sqrt{\|x - x'\|^2 + \|b-\mathcal{A}x - w'\|^2}\geq \|x-x'\|\geq
\text{dist}(x, \mathcal{G}_2(\bar{y})),
$$
we prove the conclusion of Lemma \ref{lemma:polyhdist}.
\end{proof}

The following result,    which
is partially motivated by the recent paper~\cite{zhou2015eb} and its  further development in~\cite{drusvyatskiy2016error}
for convex composite  optimization problems regularized by the nuclear norm function of rectangular matrices, provides a general approach for proving the metric subregularity of $\mathcal{T}_{\phi}$ associated with problem (\ref{eb:opt}) where the constraint $b - \mathcal{A}x\in\mathcal{Q}^\circ$ is present.
\begin{theorem}\label{prop:quadraticgrow}
Assume that $\mathcal{T}_l^{-1}(0)$  is non-empty.
Suppose that Assumption \ref{ass:f} holds and that there exists $(\bar{y}, \bar{w}, \bar{s})\in\mathcal{T}^{-1}_\psi(0)$ such that the collection of three sets $\{\overline{\mathcal{V}}, \mathcal{G}_1(\bar{y}),\mathcal{G}_2(\bar{y})\}$ is boundedly linearly regular.
 Then the  second order growth condition (\ref{defn:quadraticGrowth})
 holds at any  $\bar{x}\in\mathcal{T}_\phi^{-1}(0)$ with respect to the optimal solution set $\mathcal{T}^{-1}_{\phi}(0)$ {for}  problem (\ref{eb:opt}).
 \end{theorem}

\begin{proof}
	Let $\bar{x}\in\mathcal{T}_\phi^{-1}(0)$ be an arbitrary but fixed point.
From Assumption \ref{ass:f} (b), we know that
there exist $\kappa_1 >0$ and a neighborhood $\mathcal{U}$ of $\bar{x}$ such that
\begin{equation}\label{ineq:metricsub}
p(x)\geq p(\bar{x}) + \langle \mathcal{A}^*\bar{y} - \bar{\eta}, x - \bar{x}\rangle  + \kappa_1 \textup{dist}^2(x, (\partial p)^{-1}(\mathcal{A}^*\bar{y} - \bar{\eta})), \;\forall \,x\in\mathcal{U}.
\end{equation}
Note that $(b-\mathcal{A}\bar{x},\bar{y})\in\textup{gph}(\mathcal{N}_{\mathcal{Q}}^{-1})$ and  $\mathcal{N}_{\mathcal{Q}}(\cdot)$ is a multi-valued polyhedral function. Thus,  we can obtain from Proposition \ref{prop:polyhedral} that $\mathcal{N}_{\mathcal{Q}}(\cdot)$ is locally upper Lipschitz continuous, which further implies the metric subregularity of
$\mathcal{N}_{\mathcal{Q}}^{-1}$  at $b-\mathcal{A}\bar{x}$ for  $\bar{y}$ by definition.
{Now} by shrinking the neighborhood $\mathcal{U}$ if necessary, we know that   there exists a constant {$\kappa_1'>0$} such that
\begin{equation}\label{ineq:metricsub2}
\delta_{\mathcal{Q}^\circ}(b-\mathcal{A}x)\geq \delta_{\mathcal{Q}^\circ}(b - \mathcal{A}\bar{x}) + \langle \bar{y}, b - \mathcal{A}x - (b-\mathcal{A}\bar{x})\rangle  + \kappa'_1 \textup{dist}^2(b-\mathcal{A}x, \mathcal{N}_{\mathcal{Q}}(\bar{y})), \;\forall \,x\in\mathcal{U}.
\end{equation}
Moreover, the assumed bounded linear regularity of $\{\overline{\mathcal{V}}, \mathcal{G}_1(\bar{y}),\mathcal{G}_2(\bar{y})\}$ and  the result in Proposition \ref{prop:Omega} imply that there exist $\kappa_2>0$ and $\kappa_3 >0$, such that for any $x\in\mathcal{U}$,
\begin{equation}\label{ineq:boundedlinear}\begin{array}{ll}
\textup{dist}^2(x,  \mathcal{T}^{-1}_{\phi}(0)) &
= \textup{dist}^2\left(x, \overline{\mathcal{V}}\cap\mathcal{G}_1(\bar{y})\cap\mathcal{G}_2(\bar{y})\right)\\[2pt]
&  \leq
\kappa_2\big(\textup{dist}^2(x, \overline{\mathcal{V}}) + \textup{dist}^2(x, \mathcal{G}_1(\bar{y}))+ \textup{dist}^2(x, \mathcal{G}_2(\bar{y}))\big)\\[2pt]
& \leq  \kappa_3\big(\|\mathcal{F}x - \bar{\zeta}\|^2   + \textup{dist}^2(x, (\partial p)^{-1}(\mathcal{A}^*\bar{y} -\bar{\eta}))+\textup{dist}^2(b-\mathcal{A}x, \mathcal{N}_{\mathcal{Q}}(\bar{y}))\big),  \end{array}
\end{equation}
where in the last inequality, the first term  comes from Hoffman's error bound~\cite{hoffman1952approximate} and the third term comes from Lemma \ref{lemma:polyhdist}.
Then by Assumption \ref{ass:f} (b), shrinking $\mathcal{U}$ if necessary, we know that
 there exists $\kappa_4>0$ such that for any $x\in\mathcal{U}$,
\begin{equation}
\label{ineq:strongconvex}
\begin{array}{ll}
h(\mathcal{F}x) \geq  h(\bar{\zeta}) + \langle \nabla h(\bar{\zeta}), \mathcal{F}x - \bar{\zeta}\rangle  + \kappa_4\|\mathcal{F}x - \bar{\zeta}\|^2.
\end{array}
\end{equation}
Summing up the inequalities (\ref{ineq:metricsub}), (\ref{ineq:metricsub2}) and (\ref{ineq:strongconvex}) and  recalling that $\bar{\eta} = \mathcal{F}^*\nabla h(\mathcal{F}\bar{x}) + c$ in (\ref{defn:ubar}), we  know
{that}
for any $x\in \mathcal{U}\cap \{x\in\X\,\mid\,b- \mathcal{A}x\in\mathcal{Q}^\circ\}$,
$$
\begin{array}{rl}
\theta(x) = &\theta(x) + \delta_{\mathcal{Q}^\circ}(b-\mathcal{A}x)\\[2pt]
\geq
& \theta(\bar{x}) + \langle \mathcal{F}^*\nabla h(\bar{\zeta}) +c- \bar{\eta}, x - \bar{x}\rangle +\kappa_4\|\mathcal{F}x - \bar{\zeta}\|^2+ \kappa_1\textup{dist}^2(x, (\partial p)^{-1}(\mathcal{A}^*\bar{y} - \bar{\eta}))\\[2pt]
&+ \kappa'_1 \textup{dist}^2(b-\mathcal{A}x, \mathcal{N}_{\mathcal{Q}}(\bar{y}))\\[2pt]
\geq &  \theta(\bar{x})  + \min \{\kappa_1,\kappa_1', \kappa_4\}\left(\|\mathcal{F}x - \bar{\zeta}\|^2 + \textup{dist}^2(x, (\partial p)^{-1}(\mathcal{A}^*\bar{y} - \bar{\eta}))  + \textup{dist}^2(b-\mathcal{A}x, \mathcal{N}_{\mathcal{Q}}(\bar{y}))\right)\\[2pt]
= &  \theta(\bar{x}) + \kappa_3^{-1}\min \{\kappa_1, \kappa_1',\kappa_4\}\textup{dist}^2(x, \mathcal{T}^{-1}_{\phi}(0)),
\end{array}
$$
which shows that the second order growth condition (\ref{defn:quadraticGrowth}) holds at $\bar{x}$ with respect to  $\mathcal{T}^{-1}_{\phi}(0)$.
\end{proof}

\subsection{The metric subregularity of $\cT_{\phi}$ for SDP problems}

In this  subsection, we analyze  the metric subregularity of $\mathcal{T}_{\phi}$ associated with SDP problems 
{where}
in (\ref{eb:opt}),  $\X = \S^n$ and
$p(\cdot) = \delta_{\S_+^n}(\cdot)$.
The corresponding primal and dual forms now can be written, respectively, as
\begin{equation}\label{primalSDP}
\begin{array}{ll}
\min & h({\mathcal{F}}X) + \langle C,X\rangle + \delta_{\S_+^n}(X) \\[5pt]
\text{s.t.} & b-\mathcal{A}X\in\mathcal{Q}^{\circ}
\end{array}
\end{equation}
and
\begin{equation}\label{dualSDP}
\begin{array}{ll}
\min & \delta_{\mathcal{Q}}(y) -\langle b,y\rangle + h^{\ast}(-w) +\delta_{\S_+^n}(S)\\[5pt]
\text{s.t.} & \mathcal{A}^*y + {\mathcal{F}}^*w + S = C.
\end{array}
\end{equation}
In the following,
 we shall
 provide sufficient conditions to ensure that  assumptions made in Theorem \ref{prop:quadraticgrow} hold for the SDP problems.

Let $\overline{X}\in\S_+^n$ and $\overline{S}\in\S_+^n$ satisfy
$0\in\overline{X}+\partial \delta_{\S_+^n}(\overline{S})$, or equivalently, $\langle \overline{X},\overline{S}\rangle  = 0$.
 {Suppose} that $\overline{Z}:=\overline{X} - \overline{S}$ {has its} eigenvalues
 $\bar{\lambda}_1\geq \bar{\lambda}_2\geq \ldots\geq \bar{\lambda}_n$  being arranged in {a} non-increasing order.
Denote
\begin{equation}\label{defn:indices}
\alpha:=\{i \mid \bar{\lambda}_i>0, \; 1\leq i\leq n\},\;
\beta:=\{i \mid \bar{\lambda}_i=0, \; 1\leq i\leq n\},\;
\gamma:=\{i \mid \bar{\lambda}_i<0, \; 1\leq i\leq n\}.
\end{equation}
 Then
 there  exists an orthogonal matrix $\overline{P}\in\mathcal{O}^n$  such that
\begin{equation}\label{eig-decom}
\overline{Z} = \overline{P}\left(\begin{array}{ccc}
\overline{\Lambda}_\alpha & & \\
& 0 &\\
& & -\overline{\Lambda}_{\gamma}
\end{array}\right)\overline{P}^T, \quad
\overline{X}= \overline{P}\left(\begin{array}{ccc}
\overline{\Lambda}_\alpha & & \\
& 0 &\\
& & 0_{|\gamma|}
\end{array}\right)\overline{P}^T,
\quad  \overline{S}= \overline{P}\left(\begin{array}{ccc}
0_{|\alpha|} & & \\
& 0 &\\
& & \overline{\Lambda}_\gamma
\end{array}\right)\overline{P}^T,
\end{equation}
where
{$\overline{\Lambda}_\alpha = {\rm diag}(\bar{\lambda}_i \mid i\in \alpha) \succ 0$
 and $\overline{\Lambda}_\gamma = {\rm diag}(|\bar{\lambda}_j| \mid j\in \gamma) \succ 0$.}
Denote $\overline{P} = [\overline{P}_{\alpha}\; \overline{P}_{\beta}\; \overline{P}_\gamma]$ with $\overline{P}_\alpha\in\R^{n\times |\alpha|}$, $\overline{P}_\beta\in\R^{n\times |\beta|}$ and $\overline{P}_\gamma\in\R^{n\times |\gamma|}$.
Then we have
$$\left\{\begin{array}{rll}
\mathcal{T}_{\S_+^n}(\overline{X}) &= &\left\{H\in\S^n \,\mid\, [\overline{P}_{\beta}\;\overline{P}_{\gamma}]^TH\,[\overline{P}_{\beta}\;\overline{P}_{\gamma}]\succeq 0\right\}, \\[2pt] 
\mathcal{T}_{\S_+^n}(\overline{S}) &= &\left\{H\in\S^n \,\mid\, [\overline{P}_{\alpha}\;\overline{P}_{\beta}]^TH\,[\overline{P}_{\alpha}\;\overline{P}_{\beta}]\succeq 0\right\},\\[2pt]
\mathcal{N}_{\S_+^n}(\overline{X})  &= &\left\{H\in\S^n \,\mid\, [\overline{P}_{\beta}\;\overline{P}_{\gamma}]^TH\,[\overline{P}_{\beta}\;\overline{P}_{\gamma}]\preceq 0, \;\overline{P}_{\alpha}^TH\overline{P}=0\right\},\\
\mathcal{N}_{\S_+^n}(\overline{S}) & =& \left\{H\in\S^n \,\mid\, [\overline{P}_{\alpha}\;\overline{P}_{\beta}]^TH\,[\overline{P}_{\alpha}\;\overline{P}_{\beta}]\preceq 0, \;\overline{P}_{\gamma}^TH\overline{P}=0\right\}.
\end{array}\right.
$$
For the convenience of later discussions, we also denote
the critical cone of $\mathbb{S}_+^n$ at $\overline{S}$ associated with $\overline{X}$  as
$$\mathcal{C}_{\mathbb{S}_+^n}(\overline{S}, \overline{X}) := \mathcal{T}_{\mathbb{S}_+^n}(\overline{S})\cap \overline{X}^\perp = \{H\in\mathbb{S}^n\,\mid\, \overline{P}_{\alpha}^TH[\overline{P}_{\alpha}\;\overline{P}_{\beta}]=0, \;\overline{P}_{\beta}^TH\overline{P}_{\beta}\succeq 0\}
$$
and  the critical cone of $\mathbb{S}_+^n$ at $\overline{X}$ associated with $\overline{S}$ as
$$\mathcal{C}_{\mathbb{S}_+^n}(\overline{X}, \overline{S}) := \mathcal{T}_{\mathbb{S}_+^n}(\overline{X})\cap \overline{S}^\perp = \{H\in\mathbb{S}^n\,\mid\, [\overline{P}_{\beta}\;\overline{P}_{\gamma}]^TH\overline{P}_{\gamma}=0, \;\overline{P}_{\beta}^TH\overline{P}_{\beta}\succeq 0\}.
$$

By noting that $\partial \delta_{\S_+^n}(\overline{S}) = \mathcal{N}_{\S_+^n}(\overline{S})$, we immediate obtain the following results.

\begin{prop}\label{prop:boundedlinearsdp}
Let $\overline{S}\in\S_+^n$ and $0\in\overline{X}+\partial \delta_{\S_+^n}(\overline{S})$. Suppose that  $\overline{S}$ and $\overline{X}$  have the eigenvalue decompositions  as in (\ref{eig-decom}). Then it holds that:  \\[5pt]
(a) $\mathcal{N}_{\S_+^n}(\overline{S})$ is a polyhedral set if and only if $|\gamma| \geq n-1$; \\[5pt]
(b) $0\in \overline{X} + \textup{ri}\,(\mathcal{N}_{\S_+^n}(\overline{S}))$ if and only if $|\beta| = 0$, i.e., $\textup{rank}(\overline{X}) + \textup{rank}(\overline{S}) = n$.
\end{prop}

%
%
Next,  we shall prove the metric subregularity of $\partial \delta_{\S_+^n}(\cdot)$ and  $\partial \delta_{\S_- ^n}(\cdot)$, which is one of the key components in our subsequent analysis\footnote{This result is part of the first author's PhD thesis~\cite[Section 2.5.2]{cuiying2016}.}.

\begin{prop}\label{sdp:metricsub}
Let $\overline{S}\in\S_+^n$ and $0\in \overline{X} +  \partial \delta_{\S_+^n}(\overline{S})$.
 Then
$\partial \delta_{\S_+^n}(\cdot)$ is metrically subregular at $\overline{X}$ for $-\overline{S}$ and $\partial \delta_{\S_-^n}(\cdot)$ is metrically subregular at $-\overline{S}$ for $\overline{X}$.
\end{prop}
\begin{proof}
In the following, 	we shall prove the metric subregularity of  $\partial \delta_{\S_-^n}(\cdot)$  at $-\overline{S}$ for $\overline{X}$ and its counterpart regarding $\partial\delta_{\mathbb{S}_+^n}$ can be obtained similarly.
Without loss of generality, let  $\overline{X}$ and $\overline{S}$ have the eigenvalue decompositions as in (\ref{eig-decom}). According to Proposition \ref{pre:metricregular:thm},
in order to prove the {metric subregularity} of $\partial \delta_{\S_-^n}(\cdot)$  at $-\overline{S}$ for $\overline{X}$, it suffices to show that
 there exist a constant $\kappa>0$ and a neighborhood $\mathcal{U}$ of $\overline{S}$ such that for any $S\;\in\S_+^n\cap \mathcal{U}$,
\begin{equation}\label{pre:im:metric}
0\geq \langle \overline{X}, -S +\overline{S}\rangle + \kappa\,\textup{dist}^2(-S, (\partial\delta_{\S_-^n})^{-1}(\overline{X}))=\langle \overline{X}, -S +\overline{S}\rangle + \kappa\,\textup{dist}^2(-S, \mathcal{N}_{\S_+^n}(\overline{X})).
\end{equation}
If $|\alpha| = 0$, {then $\overline{X}=0$} and the inequality (\ref{pre:im:metric}) holds automatically for any $\kappa\geq 0$ and any neighborhood $\mathcal{U}$ of $\overline{S}$. Thus, we only need to consider the case that 
{$|\alpha|\not=0$.}
Since the case that $|\gamma|=0$ can be proved similarly as in the case for $|\gamma|\neq 0$, {we} only consider the latter case.
Set {$\rho := \frac{1}{2}\min\{1,|\bar{\lambda}_j| \mid j\in \gamma\}>0$}.
Let $S\in \S_+^n\cap\mathcal{B}_{\mathbb{S}^n}(\overline{S},\rho)$ be {arbitrarily} chosen.  We write $\widetilde{S} = \overline{P}^TS\overline{P}$
and decompose $\widetilde{S}$ into the following form:
$$
\widetilde{S}\equiv \left(\begin{array}{ccc}
\widetilde{S}_{\alpha\alpha} &\widetilde{S}_{\alpha\beta}  & \widetilde{S}_{\alpha\gamma}\\
\widetilde{S}_{\alpha\beta}^T &\widetilde{S}_{\beta\beta}  & \widetilde{S}_{\beta\gamma}\\
\widetilde{S}_{\alpha\gamma}^T &\widetilde{S}_{\beta\gamma}^T  & \widetilde{S}_{\gamma\gamma}
\end{array}\right).
$$
By the fact that $S\in\S_+^n$, we can easily check that
$$
\Pi_{\mathcal{N}_{\S_+^n}(\overline{X})}(-S)= -\overline{P}\left(\begin{array}{ccc}
0 & 0 & 0\\
0 &\widetilde{S}_{\beta\beta}  & \widetilde{S}_{\beta\gamma} \\
0 & \widetilde{S}_{\beta\gamma}^T   & \widetilde{S}_{\gamma\gamma} \end{array}\right)\overline{P}^T.
$$
Thus, we have
\begin{equation}\label{pre:metricregular:ineq00}
 \textup{dist}^2(-S, \mathcal{N}_{\S_+^n}(\overline{X}))
{=\norm{-S - \Pi_{\mathcal{N}_{\S_+^n}(\overline{X})}(-S)}^2 }
=   \|\widetilde{S}_{\alpha\alpha}\|^2+
2\|\widetilde{S}_{\alpha\beta}\|^2 + 2\|\widetilde{S}_{\alpha\gamma}\|^2.
\end{equation}
{Next we proceed to estimate $\norm{\widetilde{S}_{\alpha\alpha}},\norm{\widetilde{S}_{\alpha\beta}}$
and $\norm{\widetilde{S}_{\alpha\gamma}}$.}
By using the Bauer-Fike Theorem \cite{BauerFike60}, one obtains that for any $i=1,\dots,|\gamma|$,
{
$$
\begin{array}{ll}
{\rm dist}(\lambda_{i}(\widetilde{S}_{\gamma\gamma}), \{|\bar{\lambda}_j| \mid j\in \gamma\})
\leq  \|\widetilde{S}_{\gamma\gamma} -  \overline{\Lambda}_\gamma\|
=\norm{\overline{P}_\gamma^T S \overline{P}_\gamma - \overline{P}_\gamma^T\, \overline{S}\, \overline{P}_\gamma}
\leq \norm{S - \overline{S} } \leq \rho.
 \end{array}
$$
The above inequality further implies that
$0 < \lambda_{i}(\widetilde{S}_{\gamma\gamma})\leq
|\bar{\lambda}_{n}| + \rho
\leq  |\bar{\lambda}_{n}|+ \frac{1}{2}$ for all $i=1,\ldots,|\gamma|$. Thus,
$\widetilde{S}_{\gamma\gamma}$ is positive definite and
$ \lambda_{\max}(\widetilde{S}_{\gamma\gamma}) \leq  |\bar{\lambda}_{n}|+ \frac{1}{2}.$
Note that $\norm{\widetilde{S}_{\alpha\alpha}} \leq \rho$ and
$\norm{\widetilde{S}_{\beta\beta}} \leq \rho$ as $S\in\mathcal{B}_{\mathbb{S}^n}(\overline{S}, \rho)$.
}

{From the fact that $\widetilde{S}_{\alpha\alpha} - \widetilde{S}_{\alpha\gamma}\widetilde{S}_{\gamma\gamma}^{-1}\widetilde{S}^{T}_{\alpha\gamma} \succeq 0$
(because $\widetilde{S} \in\mathbb{S}_+^n$), we have
\begin{equation*}
\lambda_{\textup{max}}^{-1} (\widetilde{S}_{\gamma\gamma})\, \widetilde{S}_{\alpha\gamma}\widetilde{S}_{\alpha\gamma}^T
\;\preceq\;
\widetilde{S}_{\alpha\gamma}\widetilde{S}_{\gamma\gamma}^{-1}\widetilde{S}_{\alpha\gamma}^T
\;\preceq\; \widetilde{S}_{\alpha\alpha} .
\end{equation*}
Hence
\begin{equation} \label{pre:metricregular:ineq11}
\|\widetilde{S}_{\alpha\gamma}\|^2 = \text{tr}(\widetilde{S}_{\alpha\gamma}\widetilde{S}_{\alpha\gamma}^T)
\leq \text{tr}(\widetilde{S}_{\alpha\alpha})\;\lambda_{\textup{max}}(\widetilde{S}_{\gamma\gamma})
\;\leq\; \frac{|\bar{\lambda}_{n}|+ \frac{1}{2}}{\bar{\lambda}_{|\alpha|}}\langle\widetilde{S}_{\alpha\alpha}, \Lambda_\alpha\rangle.
\end{equation}
}
Moreover, we obtain from
$
\left(\begin{array}{cc}
\widetilde{S}_{\alpha\alpha} & \widetilde{S}_{\alpha\beta} \\
\widetilde{S}_{\alpha\beta}^T & \widetilde{S}_{\beta\beta}
\end{array}\right)\succeq 0
$ that
$$\widetilde{S}_{ij}^2\leq \widetilde{S}_{ii}\widetilde{S}_{jj}\leq \rho\widetilde{S}_{ii}\leq \frac{1}{2}\widetilde{S}_{ii} \leq \frac{1}{2\bar{\lambda}_{|\alpha|}}\bar{\lambda}_{i}\widetilde{S}_{ii}\,, \;\forall\, i\in\alpha,\, j\in\beta,$$
which implies that
\begin{equation}\label{pre:metricregular:ineq22}
\|\widetilde{S}_{\alpha\beta}\|^2 = \sum_{i\in\alpha, j\in\beta}\widetilde{S}_{ij}^2\leq \frac{|\beta|}{2\bar{\lambda}_{|\alpha|}}\langle \widetilde{S}_{\alpha\alpha}, \Lambda_{\alpha}\rangle.
\end{equation}
Let $\kappa:= \displaystyle\frac{\bar{\lambda}_{|\alpha|}}{2|\bar{\lambda}_{n}| + 3/2 + |\beta|}>0$. Then,
in view of (\ref{pre:metricregular:ineq00}), (\ref{pre:metricregular:ineq11}), (\ref{pre:metricregular:ineq22}) and
$$
\|\widetilde{S}_{\alpha\alpha}\|^2 \leq \rho\|\widetilde{S}_{\alpha\alpha}\|\leq \rho\;\text{tr}(\widetilde{S}_{\alpha\alpha})\leq\frac{1}{2\bar{\lambda}_{|\alpha|}}\langle \widetilde{S}_{\alpha\alpha} , \Lambda_{\alpha}\rangle,
$$
we obtain that for any $S\in \S_+^n\cap\mathcal{B}_{\mathbb{S}^n}(\overline{S},\rho)$,
$$
\begin{array}{ll}
&\quad \langle \overline{X}, -S+\overline{S}\rangle + \kappa\,\textup{dist}^2(-S, \mathcal{N}_{\S_+^n}(\overline{X}))
\;=\;
\langle \Lambda_{\alpha}, -\widetilde{S}_{\alpha\alpha}\rangle + \kappa\;(\|\widetilde{S}_{\alpha\alpha}\|^2 + 2\|\widetilde{S}_{\alpha\beta}\|^2 + 2\|\widetilde{S}_{\alpha\gamma}\|^2)
\\[8pt]
& \leq \langle \Lambda_{\alpha}, -\widetilde{S}_{\alpha\alpha}\rangle +   \kappa\left(\displaystyle\frac{1}{2\bar{\lambda}_{|\alpha|}}+ \frac{|\beta|}{\bar{\lambda}_{|\alpha|}}
+ \frac{2|\bar{\lambda}_{n}|+1}{\bar{\lambda}_{|\alpha|}}\right)\langle \widetilde{S}_{\alpha\alpha}, \Lambda_{\alpha}\rangle
=  \langle \Lambda_{\alpha},-\widetilde{S}_{\alpha\alpha}\rangle + \langle \widetilde{S}_{\alpha\alpha}, \Lambda_{\alpha}\rangle =   0.
\end{array}
$$
 Therefore, the inequality (\ref{pre:im:metric}) holds for any $S\in \S_+^n\cap\mathcal{B}_{\mathbb{S}^n}(\overline{S},\rho)$ and the proof is completed.
\end{proof}
Combining  Theorem \ref{prop:quadraticgrow} and Propositions \ref{prop:boundedlinearsdp} and \ref{sdp:metricsub}, we obtain the following result.
\begin{corollary}\label{coro:sdpTphi}
Suppose that  Assumption \ref{ass:f}\,(a) holds for problem (\ref{primalSDP}).
 Then $\mathcal{T}_{\phi}$ is {metrically} subregular at any ${X}\in\mathcal{T}_{\phi}^{-1}(0)$ for the origin under one of the following two conditions:\\[5pt]
  (i) there exists $(\bar{y}, \bar{w}, \overline{S})\in\mathcal{M}_{\phi}(\overline{X})$ such that $\textup{rank}(\overline{S})\geq n-1$;\\[5pt]
 (ii) there exist $\overline{X}\in\mathcal{T}_{\phi}^{-1}(0)$ and $(\bar{y}, \bar{w}, \overline{S})\in\mathcal{M}_{\phi}(\overline{X})$ such that $\textup{rank}(\overline{X}) + \textup{rank}(\overline{S}) = n$.
\end{corollary}

\subsection{The metric subregularity of $\mathcal{T}_l$ for SDP problems}

In this subsection, we focus on the metric subregularity of $\mathcal{T}_l$ at a KKT point for the origin
 associated with the SDP problem (\ref{primalSDP}) and its dual
 (\ref{dualSDP}).
 Denote $\E:=\Y\times \W\times \S^n\times \S^n$.
  Consider a perturbed point $(u,V)\in\E$ with
 $u: = (u_1, u_2, U)\in\Y\times \W\times \S^n$ and $V\in\S^n$. Then
 $(y,w,S,X)\in\mathcal{T}_l^{-1}(u,V)$ if and only if $(y,w,S,X)\in\E$ solves the following perturbed KKT system:
  \begin{equation}\label{pertkkt}
\left\{\begin{array}{ll}
u_1\in -b + \mathcal{A}{X} +\mathcal{N}_{\mathcal{Q}}(y),\quad
u_2\in \mathcal{F}X - \partial h^*(-w),\quad
U\in X + \mathcal{N}_{\S_+^n}(S),\\[5pt]
V = C - (\mathcal{A}^*y + \mathcal{F}^*w + S),
\end{array}\right. \quad (y,w,S,X)\in\E.
\end{equation}


 We know from Corollary \ref{coro:sdpTphi} that for problem  (\ref{primalSDP}), if there exists a KKT point satisfying the partial strict complementarity condition as in (ii) of Corollary \ref{coro:sdpTphi},  then
 $\mathcal{T}_\phi$ is metrically subregular at any $\overline{X}\in\mathcal{T}_{\phi}^{-1}(0)$ for the origin. Naturally
one may ask  whether a  corresponding result can be extended to $\mathcal{T}_l$ under the same assumptions. The answer is unfortunately negative, as can be seen from the following example.

\begin{example}\label{ex:1}
Consider the following SDP problem:
\begin{equation}\label{ex1}
\begin{array}{ll}
\min \Big\{  X_{22} + \delta_{\S_+^2}(X)  
\mid  2X_{12} - X_{22}  = 0 \Big\}
\end{array}
\end{equation}
and its dual (in its equivalent minimization format):
\begin{equation}\label{ex:dual}
\begin{array}{ll}
\min \Big\{ \delta_{\S_+^2}(S) \;\mid\;
S_{11} = 0,\; y+S_{12} =0,\; -y+S_{22} = 1 \Big\}.
\end{array}
\end{equation}
It is easy to check that {the sets of optimal solutions to (\ref{ex1}) and (\ref{ex:dual}) are given, respectively, by}
$$
\mathcal{T}_{\phi}^{-1}(0) = \left\{X\in\S_+^2\,\mid\,  X_{11}\geq 0,\;X_{22} = 0\right\},\;
\mathcal{T}_{\psi}^{-1}(0) = \left\{(y,S)\in\R\times \S_+^2\,\mid\, y = 0, \; S_{11} =  0, \;S_{22}= 1
\right\}.
$$
In addition, for any $(\bar{y},\overline{S},\overline{X})\in\mathcal{T}_l^{-1}(0) = \mathcal{T}_{\psi}^{-1}(0)\times \mathcal{T}_{\phi}^{-1}(0)$ with $\overline{X}_{11}>0$, it holds that
$\textup{rank}(\overline{S}) = 1$ and
$\textup{rank}(\overline{X}) + \textup{rank}(\overline{S}) = 2$. Therefore, we know from Corollary \ref{coro:sdpTphi} that
 $\mathcal{T}_{\phi}$ is metrically subregular at any $\overline{X}\in\mathcal{T}_{\phi}^{-1}(0)$ for the origin.
However, $\mathcal{T}_l$ fails to be metrically subregular at $(\bar{y},\overline{S},\overline{X}')\in\mathcal{T}_l^{-1}(0)$ with $\overline{X}'_{11} =0$. This can be seen as follows: for any $\epsilon>0$,
consider the perturbed points $u(\epsilon) := \left(0, \; \begin{pmatrix}
\epsilon & \epsilon \\  \epsilon & 2\epsilon
 \end{pmatrix}\right)\in\R\times \S^2
 $ and $V(\epsilon) := \begin{pmatrix}
-\epsilon & 0 \\   0 & 0
 \end{pmatrix}\in\S^2$.
  Then
one can {show} from   (\ref{pertkkt}) that
$$
(y(\epsilon), S(\epsilon),X(\epsilon)) := \left(
\sqrt{\epsilon},\;
\begin{pmatrix}
\epsilon & -\sqrt{\epsilon} \\
-\sqrt{\epsilon} & 1+\sqrt{\epsilon}
\end{pmatrix},\;\begin{pmatrix}
\epsilon & \epsilon \\
\epsilon & 2\epsilon
\end{pmatrix}
\right)
\in\mathcal{T}_l^{-1}(u(\epsilon), V(\epsilon)).$$
Also one can readily verify that $$\begin{array}{cc}
\|(u(\epsilon), V(\epsilon))\| = 2\sqrt{2}\epsilon, \; \textup{dist}(X(\epsilon), \mathcal{T}_\phi^{-1}(0)) = \left\|
\begin{pmatrix}
0 & \epsilon \\
\epsilon &  2\epsilon
\end{pmatrix}
\right\| = \sqrt{6}\epsilon,\\[5pt]
\textup{dist}((y(\epsilon), S(\epsilon)), \mathcal{T}_\psi^{-1}(0)) = \left( (\sqrt{\epsilon})^2 + \left\|
\begin{pmatrix}
\epsilon & -\sqrt{\epsilon} \\
-\sqrt{\epsilon} &  \sqrt{\epsilon}
\end{pmatrix}
\right\|^2 \right)^{1/2}\geq 2\sqrt{\epsilon}.
\end{array}
$$
Thus, there cannot exist a constant $\kappa\geq 0$ along with a neighborhood $\mathcal{U}$ of $(\bar{y}, \overline{S},\overline{X}')$ such that
$$
\textup{dist}((y, S,X), \mathcal{T}_l^{-1}(0)) \leq \kappa \|(u, V)\|,\quad \forall \; (y,S,X)\in\mathcal{T}_l^{-1}(u, V)\cap\mathcal{U}.$$
\end{example}
Example \ref{ex:1} shows that very different from the case for $\mathcal{T}_{\phi}$, the operator $\mathcal{T}_l$ may fail to be  metrically subregular at a KKT point for the origin
under either condition (i) or condition (ii) in Corollary \ref{coro:sdpTphi}. Therefore, additional conditions must be imposed in order to guarantee the metric subregularity of $\mathcal{T}_l$.
For this purpose, we shall first introduce a second order sufficient condition for problem  (\ref{dualSDP}).
The following 
assumption on $h^*$ is made throughout this {subsection.}
\begin{assumption}\label{assump:hstar}
The function
$h^*$ is  continuously differentiable  on $\textup{dom}(h^*)$ and its gradient is
directionally differentiable at every point in  $\textup{dom}(h^*)$.
\end{assumption}
Let $(\bar{y}, \bar{w}, \overline{S}, \overline{X})\in\mathcal{T}_l^{-1}(0)$.
The critical cone of problem (\ref{dualSDP}) at $(\bar{y}, \bar{w}, \overline{S})$  is defined by
$$\begin{array}{ll}
\mathcal{C}_{\psi}(\bar{y}, \bar{w}, \overline{S}) :=  &\bigg\{(d_y,d_w,d_S)\in\mathbb{Y}\times \mathbb{W}\times \mathbb{S}^n\,\mid\, \mathcal{A}^*d_y + \mathcal{F}^*d_w + d_S = 0,\;d_y\in\mathcal{T}_{\mathcal{Q}}(\bar{y}),\; d_S\in\mathcal{T}_{\mathbb{S}_+^n}(\overline{S}), \\[5pt]
& \qquad \qquad \qquad\qquad\qquad\qquad\quad \langle -b, d_y\rangle +\langle \nabla h^*(-\bar{w}), d_w\rangle = 0\bigg\}.
\end{array}
$$
For any given $S \in \mathbb{S}^n$, define the linear-quadratic function $\Upsilon_S: \mathbb{S}^n \times  \mathbb{S}^n \to \mathbb{R}$, which is
linear in the first argument and quadratic in the second argument, by
$$\Upsilon_S(\Gamma,D):=2 \langle \Gamma,DS^\dagger D\rangle,\; \forall\,(\Gamma,D)\in\mathbb{S}^n\times \mathbb{S}^n,
$$
where $S^\dagger$ is the Moore-Penrose pseudo-inverse of $S$.
For  problem (\ref{dualSDP}), we say {that} the  second order sufficient condition
 holds at $(\bar{y}, \bar{w}, \overline{S})$ with respect to the multiplier $\overline{X}\in\mathcal{M}_\psi(\bar{y}, \bar{w}, \overline{S})$  if
\begin{equation}\label{sosc:dual}
\langle d_w,  (\nabla h^*)'(-\bar{w};d_w)\rangle + \Upsilon_{\overline{S}}(\overline{X}, d_S) > 0, \;\forall\, 0\neq (d_y,d_w, d_S)\in\mathcal{C}_{\psi}(\bar{y}, \bar{w},\overline{S}).
\end{equation}
Here we only {deal with} the inequality (\ref{sosc:dual}) at
a particular point $\overline{X}\in\mathcal{M}_{\psi}(\bar{y}, \bar{w}, \overline{S})$
instead of taking {``sup"} over the  set $\mathcal{M}_{\psi}(\bar{y}, \bar{w}, \overline{S})$ of multipliers,  where the latter one is adopted in~\cite[Theorem 3.86]{bonnans2013perturbation}. Thus, it is a more 
{restrictive}
condition even if $h$ is assumed to be twice continuously differentiable.

The primary motivation to employ the second order sufficient condition (\ref{sosc:dual}) for studying the metric subregularity  comes from the  recent advances on  the complete characterization of the metric subregularity of $\mathcal{T}_l$ for the convex QSDP problem,  when the KKT solution point is
assumed to be unique~\cite{han2015linear}.
Define the KKT mapping $\mathcal{R}: \E\to\E$ associated with problem (\ref{dualSDP}) as
$$
\mathcal{R}(y,w,S,X):=\left(\begin{array}{cc}
y-\Pi_{\mathcal{Q}}(y + b-\mathcal{A}X)\\[2pt]
\mathcal{F}X + \nabla h^*(-{w}) \\[2pt]
S - \Pi_{\S_+^n}(S-X)\\[2pt]
C - (\mathcal{A}^*y + \mathcal{F}^*w + S)
\end{array}
\right),\;\forall \, (y,w,S,X)\in\mathbb{E}.
$$
It has been shown in~\cite[Theorem 5.1]{han2015linear} that for linear  and least squares SDP problems,
when $(\bar{y}, \bar{w}, \overline{S}, \overline{X})\in\mathbb{E}$ is the unique solution to $\mathcal{R}({y},{w}, {S}, {X}) = 0$, then
$\mathcal{R}$ is metrically subregular at $(\bar{y}, \bar{w}, \overline{S}, \overline{X})$ for the origin if and only if both the primal and dual second order sufficient conditions hold\footnote{In fact, this characterization of metric subregularity for $\mathcal{T}_l$ is also true for the convex QSDP problem even if  its least squares representation is not explicitly 
 available computationally. For details, see~\cite{han2015linear}.}.
On the other hand, Proposition \ref{prop:equivmetric} says that
the metric subregularity of $\mathcal{R}$  at $(\bar{y}, \bar{w}, \overline{S}, \overline{X})$ for the origin is equivalent to the metric subregularity of $\mathcal{T}_l$  at $(\bar{y}, \bar{w}, \overline{S}, \overline{X})$ for the origin. Thus, we consider the second order sufficient condition (\ref{sosc:dual}) for problem (\ref{dualSDP}) when the solution set to problem (\ref{primalSDP}) is not  a singleton.
We need two perturbation properties, one  on the SDP cone and the other on the  polyhedral cone,  before stating our main result on the metric subregularity of $\mathcal{T}_l$.

\begin{prop}\label{relationSX}
Let $\overline{S}\in\S_+^n$ and $0\in \overline{X} +  \partial \delta_{\S_+^n}(\overline{S})$. Suppose that $\overline{X}$ and $\overline{S}$ have the eigenvalue decompositions as in (\ref{eig-decom}).
Then for all $(X,S)\in\S^n\times \S^n$ satisfying $0\in X + \partial\delta_{\S_+^n}(S)$ {and is}
 sufficiently close to $(\overline{X},\overline{S})\in\S^n\times \S^n$, we have
\begin{equation}\label{prop:tildex}
\left\{\begin{array}{ll}
\widetilde X_{\alpha\alpha} = \overline{\Lambda}_{\alpha}  + O(\|\Delta X\|), \quad \widetilde X_{\alpha\beta} =  O(\|\Delta X\|),\quad
\widetilde X_{\alpha\gamma} = O(\min\{\|\Delta X\|, \|\Delta S\|\}),\\[5pt]
\widetilde X_{\beta\beta} = O(\|\Delta X\|), \quad  \widetilde X_{\beta\gamma} = O(\|\Delta X\|\|\Delta S\|),\quad \widetilde X_{\gamma\gamma} = O(\|\Delta X\|\|\Delta S\|),\\[5pt]
\widetilde S_{\alpha\alpha} = O(\|\Delta X\|\|\Delta S\|), \quad \widetilde S_{\alpha\beta} =  O(\|\Delta X\|\|\Delta S\|),\quad
\widetilde S_{\alpha\gamma} = O(\min\{\|\Delta X\|, \|\Delta S\|\}),\\[5pt]
\widetilde S_{\beta\beta} = O(\|\Delta S\|), \quad \widetilde S_{\beta\gamma} = O(\|\Delta S\|),\quad \widetilde S_{\gamma\gamma} ={\overline{\Lambda}_{\gamma}} + O(\|\Delta S\|),
\end{array}\right.
\end{equation}
\begin{equation}\label{prop:sigmaterm}
\widetilde{S}_{\alpha\gamma}  + \overline{\Lambda}^{-1}_{\alpha}\widetilde{X}_{\alpha\gamma}\overline{\Lambda}_{\gamma} = O(\|\Delta X\|\|\Delta S\|),\\[3pt]
\end{equation}
\begin{equation}\label{prop:betabeta}
\langle \widetilde{X}_{\beta\beta}, \widetilde{S}_{\beta\beta}\rangle  =
\left\{\begin{array}{ll}
O(\|\Delta X\|\|\Delta S\|)(\|\Delta X\| + \|\Delta S\|) & \textup{if}\quad |\alpha|>0, \\[5pt]
O(\|\Delta X\|\|\Delta S\|^2) &  \textup{if}\quad |\alpha|=0,
\end{array}\right.
\end{equation}
where $\Delta X := X - \overline{X}$, $\Delta S := S - \overline{S}$,
$
\widetilde{X}:= \overline{P}^TX\overline{P}$ and  $\widetilde{S}:= \overline{P}^TS\overline{P}.
$
\end{prop}
\begin{proof}
 Let $\mu_1 > \ldots>{\mu_{r}} >0$ and $0< \nu_{_1} < \ldots <{\nu_{s}}$ be all the distinct eigenvalues of $\overline{X}$ and $\overline{S}$, respectively.
Denote
$$\alpha_{i}:=\{\,k\in\alpha \mid\bar{\lambda}_k = \mu_{i}\},\;i=1,\ldots, {r},\quad\quad
\gamma_j:=\{\,k\in\gamma\mid\bar{\lambda}_k= -\nu_{{j}}\}, \; j=1,\ldots,{s}.
$$
It is easy to see that for all $\Delta X$ and $\Delta S$ sufficiently small, there exists $\widetilde{P}\in\mathcal{O}^n$ such that
$$
\widetilde{X} = \widetilde{P}
\begin{pmatrix}
{\Lambda}_{\alpha} & & \\
& \Lambda_{\beta} & \\
&& 0_{|\gamma|}
\end{pmatrix}
\widetilde{P}^T,\quad\quad  \widetilde{S} = \widetilde{P}
\begin{pmatrix}
{0}_{|\alpha|} & & \\
& \Lambda'_{\beta} & \\
&& {\Lambda_{\gamma}}
\end{pmatrix}
\widetilde{P}^T,
$$
where $\Lambda_{\alpha}\succ 0$, $\Lambda_{\beta}\succeq 0$, $\Lambda'_{\beta}\succeq 0$, $\Lambda_{\gamma}\succ 0$ and
$\langle \Lambda_{\beta}, \Lambda_{\beta}'\rangle =0$.
From~\cite[Lemma 4.12]{sun2002semismooth},  for all  $\Delta X$ and $\Delta S$ sufficiently small, there exist $\Theta_{\alpha_i}\in\mathcal{O}^{|\alpha_i|}, \; i=1, \ldots,r$, $\Theta'\in\mathcal{O}^{|\beta| + |\gamma|}$,  $\Theta_{\gamma_i}\in\mathcal{O}^{|\gamma_i|}, \; i=1, \ldots,s$ and  $\Theta''\in\mathcal{O}^{|\alpha| + |\beta|}$ such that
\begin{equation}\label{tildeP}
\widetilde{P} =
\begin{pmatrix}
\Theta_{\alpha} & \\
& \Theta'
\end{pmatrix}
+ O(\|\Delta X\|)
 =
\begin{pmatrix}
\Theta{''} &\\
&\Theta_{\gamma}
\end{pmatrix} + O(\|\Delta S\|),
\end{equation}
where $\Theta_{\alpha}\in\mathcal{O}^{|\alpha|}$ and $\Theta_{\gamma}\in\mathcal{O}^{|\gamma|}$ are block-diagonal orthogonal matrices given by
$$\Theta_{\alpha} := \begin{pmatrix}
\Theta_{\alpha_1} &&\\
& \ddots &\\
&&\Theta_{\alpha_{r}}\\
\end{pmatrix}, \quad \quad
\Theta_{\gamma} := \begin{pmatrix}
\Theta_{\gamma_1} &&\\
& \ddots &\\
&&\Theta_{\gamma_{s}}
\end{pmatrix}.
$${
Note that
\begin{equation} \label{Q1}
\Theta_\alpha^T \overline{\Lambda}_\alpha \Theta_\alpha =  \overline{\Lambda}_\alpha, \quad
\Theta_\gamma^T \overline{\Lambda}_\gamma \Theta_\gamma =  \overline{\Lambda}_\gamma.
\end{equation}
By using (\ref{tildeP}) and the fact that for any $N
\in\R^{|\beta|\times |\beta|}$,
$$ NN^T = I_{|\beta|} + O(\|\Delta X\|+\|\Delta S\|) \Longrightarrow
\exists\; \widehat{N}\in\mathcal{O}^{|\beta|}\;\text{such that}\; \widehat{N} = N + O(\|\Delta X\|+\|\Delta S\|),
$$
we further obtain from (\ref{defn:tildeP}) that there exists $\Theta_{\beta}\in\mathcal{O}^{|\beta|}$ such that
\begin{equation}\label{defn:tildeP}
\widetilde{P}
=
\begin{pmatrix}
\Theta_{\alpha} + O(\|\Delta X\|) & O(\|\Delta X\|) & \widetilde{P}_{\alpha\gamma}\\[5pt]
O(\|\Delta X\|) & \Theta_{\beta} + O(\|\Delta X\| + \|\Delta S\|)  & O(\|\Delta S\|)\\[5pt]
\widetilde{P}_{\gamma\alpha} & O(\|\Delta S\|) &     \Theta_{\gamma} + O(\|\Delta S\|)
\end{pmatrix}
\end{equation}
with
\begin{equation}\label{defn:widePalphagamma}
\widetilde{P}_{\alpha\gamma} = O(\min\{\|\Delta X\|, \|\Delta S\|\}) ,\quad \quad \widetilde{P}_{\gamma\alpha} = O(\min\{\|\Delta X\|, \|\Delta S\|\}).
\end{equation}
It then follows from (\ref{defn:tildeP}), (\ref{defn:widePalphagamma}) and the orthogonality of
 $\widetilde{P}$  that for all $\Delta X$ and $\Delta S$ sufficiently small,
\begin{equation}\label{sigmaterm2}
\begin{array}{ll}
 0 & = \widetilde{P}_{\alpha\alpha}\widetilde{P}_{\gamma\alpha}^T + \widetilde{P}_{\alpha\beta}\widetilde{P}_{\gamma\beta}^T+ \widetilde{P}_{\alpha\gamma}\widetilde{P}_{\gamma\gamma}^T \\[5pt]
 & = (\Theta_{\alpha} + O(\|\Delta X\|))\widetilde{P}_{\gamma\alpha}^T + O(\|\Delta X\|\|\Delta S\|) + \widetilde{P}_{\alpha\gamma}(\Theta_{\gamma}^T + O(\|\Delta S\|))\\[5pt]
 & = \Theta_{\alpha}\widetilde{P}_{\gamma\alpha}^T + \widetilde{P}_{\alpha\gamma} \Theta^T_{\gamma}+ O(\|\Delta X\|\|\Delta S\|),\\[5pt]
 \end{array}
\end{equation}
\begin{equation}\label{betabeta2}
\left\{\begin{array}{ll}
\widetilde{P}^T_{\beta\beta}\widetilde{P}_{\beta\alpha} = -\widetilde{P}_{\alpha\beta}^T\widetilde{P}_{\alpha\alpha} - \widetilde{P}_{\gamma\beta}^T\widetilde{P}_{\gamma\alpha}  = O(\|\Delta S\|\|\Delta X\|) + O(\|\Delta X\|),\\[5pt]
\widetilde{P}^T_{\beta\beta}\widetilde{P}_{\beta\beta} =I_{|\beta|}-\widetilde{P}_{\alpha\beta}^T\widetilde{P}_{\alpha\beta} - \widetilde{P}_{\gamma\beta}^T\widetilde{P}_{\gamma\beta }  = \left\{\begin{array}{ll} I_{|\beta|} + O(\|\Delta X\|^2+\|\Delta S\|^2) & \textup{if}\quad |\alpha|>0,\\[5pt]
I_{|\beta|} + O(\|\Delta S\|^2) & \textup{if}\quad |\alpha|=0,\end{array}\right.
\\[5pt]
\widetilde{P}^T_{\beta\gamma}\widetilde{P}_{\beta\alpha} = O(\|\Delta X\|\|\Delta S\|),\\[5pt]
\widetilde{P}^T_{\beta\gamma}\widetilde{P}_{\beta\beta} = -\widetilde{P}_{\alpha\gamma}^T\widetilde{P}_{\alpha\beta} - \widetilde{P}_{\gamma\gamma}^T\widetilde{P}_{\gamma\beta}  = \left\{\begin{array}{ll}
O(\|\Delta X\|\|\Delta S\|) + O(\|\Delta S\|)& \textup{if}\quad |\alpha|>0,\\[5pt]
 O(\|\Delta S\|) & \textup{if}\quad |\alpha|=0.
\end{array}\right.
\end{array}\right.
\end{equation}
}
By using (\ref{tildeP})-(\ref{defn:widePalphagamma}), the definitions of $\widetilde{X}$ and $\widetilde{S}$ and the Bauer-Fike Theorem \cite{BauerFike60},
we obtain that for all $\Delta X$ and $\Delta S$ sufficiently small,
\begin{equation}\label{tildex}
\left\{\begin{array}{ll}
\widetilde X_{\alpha\alpha} 
 =\overline{\Lambda}_{\alpha} + O(\|\Delta X\|), \quad \widetilde X_{\alpha\beta} =  O(\|\Delta X\|),\\[5pt]
\widetilde X_{\alpha\gamma}
 =  \overline{\Lambda}_{\alpha}\Theta_{\alpha}\widetilde{P}_{\gamma\alpha}^T + O(\|\Delta X\|\|\Delta S\|) = O(\min\{\|\Delta X\|, \|\Delta S\|\}),\\[5pt]
\widetilde X_{\beta\beta} = O(\|\Delta X\|), \quad  \widetilde X_{\beta\gamma} = O(\|\Delta X\|\|\Delta S\|),\quad \widetilde X_{\gamma\gamma} = O(\|\Delta X\|\|\Delta S\|),
\end{array}\right.
\end{equation}
\begin{equation}\label{tildeS}
\left\{\begin{array}{ll}
\widetilde S_{\alpha\alpha} = O(\|\Delta X\|\|\Delta S\|), \quad \widetilde S_{\alpha\beta} =  O(\|\Delta X\|\|\Delta S\|), \\[5pt]
\widetilde S_{\alpha\gamma} 
 ={\widetilde{P}_{\alpha\gamma}\Theta_\gamma^T}\overline{\Lambda}_{\gamma}
 + O(\|\Delta X\|\|\Delta S\|) = O(\min\{\|\Delta X\|, \|\Delta S\|\}),\\[5pt]
 \widetilde S_{\beta\beta} = O(\|\Delta S\|),\quad \widetilde S_{\beta\gamma} = O(\|\Delta S\|),\quad \widetilde S_{\gamma\gamma}
  =\overline{\Lambda}_{\gamma} + O(\|\Delta S\|),
\end{array}\right.
\end{equation}
which
show that (\ref{prop:tildex}) holds.

Next, we shall  prove (\ref{prop:sigmaterm}) and (\ref{prop:betabeta}).
In view of (\ref{sigmaterm2}), (\ref{tildex}) and  (\ref{tildeS}),  we know that for all $\Delta X$ and $\Delta S$ sufficiently small,
$$
\widetilde S_{\alpha\gamma}  = {\widetilde{P}_{\alpha\gamma} \Theta_{\gamma}^T\overline{\Lambda}_{\gamma}}
+ O(\|\Delta X\|\|\Delta S\|)
 = -\Theta_{\alpha}\widetilde{P}_{\gamma\alpha}^T \overline{\Lambda}_{\gamma}+ O(\|\Delta X\|\|\Delta S\|)= 
-\overline{\Lambda}_{\alpha}^{-1}\widetilde{X}_{\alpha\gamma} \overline{\Lambda}_{\gamma}+ O(\|\Delta X\|\|\Delta S\|).
$$
Finally,
we conclude from
$\Lambda_{\beta} = O(\|\Delta X\|)$, $\Lambda_{\beta}' = O(\|\Delta S\|)$, $\langle \Lambda_{\beta},\Lambda_{\beta}'\rangle  = 0$ and  (\ref{betabeta2}) that for all $\Delta X$ and $\Delta S$ sufficiently small and $|\alpha|>0$,
$$
\begin{array}{ll}
&
\langle \widetilde{X}_{\beta\beta}, \widetilde{S}_{\beta\beta}\rangle
 \;=\;  \langle \widetilde{P}_{\beta\alpha}\Lambda_{\alpha}\widetilde{P}_{\beta\alpha}^T + \widetilde{P}_{\beta\beta}\Lambda_{\beta}\widetilde{P}_{\beta\beta}^T,\, \widetilde{P}_{\beta\beta}\Lambda_{\beta}'\widetilde{P}_{\beta\beta}^T + \widetilde{P}_{\beta\gamma}\Lambda_{\gamma}\widetilde{P}_{\beta\gamma}^T\rangle
\\[8pt]
= &\langle \widetilde{P}^T_{\beta\beta}\widetilde{P}_{\beta\alpha}\Lambda_{\alpha},\, \Lambda_{\beta}'\widetilde{P}^T_{\beta\beta}\widetilde{P}_{\beta\alpha}\rangle + \langle \widetilde{P}^T_{\beta\beta}\widetilde{P}_{\beta\beta}\Lambda_{\beta},\, \Lambda_{\beta}'\widetilde{P}^T_{\beta\beta}\widetilde{P}_{\beta\beta}\rangle   + \langle \widetilde{P}_{\beta\gamma}^T\widetilde{P}_{\beta\alpha}\Lambda_{\alpha},\, \Lambda_{\gamma}\widetilde{P}^T_{\beta\gamma}\widetilde{P}_{\beta\alpha}\rangle \\[5pt]
& +\langle \widetilde{P}^T_{\beta\gamma}\widetilde{P}_{\beta\beta}\Lambda_{\beta},\, \Lambda_{\gamma}\widetilde{P}^T_{\beta\gamma}\widetilde{P}_{\beta\beta}\rangle
\\[8pt]
= & O(\|\Delta X\|^2\|\Delta S\|) + O(\|\Delta X\|\|\Delta S\|)(\|\Delta X\|^2 + \|\Delta S\|^2) + O(\|\Delta X\|^2\|\Delta S\|^2) + O(\|\Delta X\|\|\Delta S\|^2)\\[5pt]
= & O(\|\Delta X\|\|\Delta S\|)(\|\Delta X\| + \|\Delta S\|)
\end{array}
$$
or for all $\Delta X$ and $\Delta S$ sufficiently small and $|\alpha|=0$,
$$
\langle \widetilde{X}_{\beta\beta}, \widetilde{S}_{\beta\beta}\rangle
=  \langle \widetilde{P}^T_{\beta\beta}\widetilde{P}_{\beta\beta}\Lambda_{\beta}, \Lambda_{\beta}'\widetilde{P}^T_{\beta\beta}\widetilde{P}_{\beta\beta}\rangle
+\langle \widetilde{P}^T_{\beta\gamma}\widetilde{P}_{\beta\beta}\Lambda_{\beta}, \Lambda_{\gamma}\widetilde{P}^T_{\beta\gamma}\widetilde{P}_{\beta\beta}\rangle= O(\|\Delta X\|\|\Delta S\|^2).
$$
This completes the proof of the proposition.
\end{proof}


Now we shall focus on the convex polyhedral cone $\mathcal{Q}$. Without loss of generality, let $\mathbb{Y} := \mathbb{R}^m$.
Let $(\bar{y}, \bar{q})\in\mathbb{R}^m\times\mathbb{R}^m$ satisfy $\bar{q}\in\mathcal{N}_{\mathcal{Q}}(\bar{y})$. We denote the critical cone of $\mathcal{Q}$ at $\bar{y}$ associated with $\bar{q}$ and the critical cone of $\mathcal{Q}^\circ$ at $\bar{q}$ associated with $\bar{y}$ as
$$\mathcal{C}_{\mathcal{Q}}(\bar{y},\bar{q}):= \mathcal{T}_{\mathcal{Q}}(\bar{y})\cap \bar{q}^\perp, \quad\quad\mathcal{C}_{\mathcal{Q}^\circ}(\bar{q},\bar{y}):= \mathcal{T}_{\mathcal{Q}^\circ}(\bar{q})\cap \bar{y}^\perp.
$$
It is easy to check  the following relation:
\begin{equation}\label{eq:polycritical}
(\mathcal{C}_{\mathcal{Q}}(\bar{y}, \bar{q}) )^\circ= \mathcal{C}_{\mathcal{Q}^\circ}(\bar{q}, \bar{y}).
\end{equation}
For subsequent discussions, we write the convex polyhedral cone $\mathcal{C}_{\mathcal{Q}}(\bar{y},\bar{q})$ as
\begin{equation}\label{eq:defnB}
\mathcal{C}_{\mathcal{Q}}(\bar{y},\bar{q}) =  \{y\in\mathbb{R}^m\,\mid\, Qy\leq 0\},
\end{equation}
where $Q$ is  some matrix in $\mathbb{R}^{r\times m}$. 
Define a collection of index sets
\begin{equation}\label{defn:IB}
\mathcal{I}_{Q}:=\{a\subseteq\{1,2,\ldots, r\}\,\mid\, \exists\;
y\in\mathbb{R}^m\; \textup{satisfying} \;{Q}_iy = 0,\; \forall\, i\in a\;\text{and}\;{Q}_iy < 0,\; \forall\, i\notin a
\}
\end{equation}
{where $Q_i$ denotes the $i$th row of $Q$.}
Moreover, for each $a\in\mathcal{I}_{Q}$, define a subspace $L_a\in\mathbb{R}^m$ by
\begin{equation}\label{defn:LI}
L_a:=\{y\in\mathbb{R}^m\,\mid\, Q_iy = 0, \;\forall\, i\in a\}.
\end{equation}

\begin{prop}\label{prop:polycritical}
Let $(\bar{y}, \bar{q})\in\mathbb{R}^m\times \mathbb{R}^m$ satisfy $\bar{q}\in\mathcal{N}_{\mathcal{Q}}(\bar{y})$ and suppose that $\mathcal{C}_{\mathcal{Q}}(\bar{y},\bar{q})$ has the form (\ref{eq:defnB}).
Then
 for any
$(y,q)\in\mathbb{R}^m\times \mathbb{R}^m$ satisfying ${q}\in\mathcal{N}_{\mathcal{Q}}({y})$
{and is}
 sufficiently close to $(\bar{y}, \bar{q})$, there exists  $a\in\mathcal{I}_{{Q}}$ such that
$$\mathcal{C}_{\mathcal{Q}}(\bar{y},\bar{q})\cap L_a\ni (y - \bar{y})\perp (q-\bar{q}) \in \mathcal{C}_{\mathcal{Q}^\circ}(\bar{q},\bar{y})\cap L_a^\perp.
$$
\end{prop}
\begin{proof}
Note that for any $(q,y)\in\mathbb{R}^m\times \mathbb{R}^m$, the relation
${q}\in\mathcal{N}_{\mathcal{Q}}({y})$ can be equivalently written as
$
 y = \Pi_{\mathcal{Q}}(y+q).
$
Since $\mathcal{Q}$ is a convex polyhedral cone, by~\cite[Theorem 4.1.1]{facchinei2007finite} we know that for all $(y,q)$ satisfying ${q}\in\mathcal{N}_{\mathcal{Q}}({y})$ {and is}  sufficiently close to $(\bar{y},\bar{q})$,
\begin{equation}\label{eq:polycri1}
\begin{array}{ll}
y = \Pi_{\mathcal{Q}}(y + q) = \Pi_{\mathcal{Q}}(\bar{y} + \bar{q}) + \Pi_{\mathcal{Q}}'(\bar{y} +\bar{q}; y-\bar{y} + q - \bar{q})= \bar{y} + \Pi_{\mathcal{C}_{\mathcal{Q}}(\bar{y},\bar{q})}\;(y -\bar{y}+ q - \bar{q}).
\end{array}
\end{equation}
It is also known from~\cite[Proposition 4.1.9]{facchinei2007finite} that
$
\Pi_{\mathcal{C}_{\mathcal{Q}}(\bar{y}, \bar{q})} = {\{\Pi_{L_a} \mid a\in\mathcal{I}_{Q}}\}
$. Then for all $(y,q)$ satisfying $q\in\mathcal{N}_{\mathcal{Q}}(y)$ sufficiently close to 
$(\bar{y}, \bar{q})$,
there exists $a\in\mathcal{I}_{Q}$ such that
$$\Pi_{\mathcal{C}_{\mathcal{Q}}(\bar{y}, \bar{q})}(y-\bar{y} + q-\bar{q}) = \Pi_{L_a}(y-\bar{y} + q-\bar{q}).$$
Thus,  the equation (\ref{eq:polycri1}) is equivalent to
$$\mathcal{C}_{\mathcal{Q}}(\bar{y},\bar{q})\cap L_a\;\ni\; (y - \bar{y})\;\perp\; (q-\bar{q})\; \in \;(\mathcal{C}_{\mathcal{Q}}(\bar{y},\bar{q})\cap L_a)^\circ,
$$
which, together with (\ref{eq:polycritical}), 
{completes the proof}
of this proposition.
\end{proof}

Suppose that the KKT solution set to problem (\ref{dualSDP}) is non-empty.
Consider an optimal solution $(\bar{y}, \bar{w}, \overline{S})$ to problem (\ref{dualSDP}) and $\overline{X}\in\mathcal{M}_\psi(\bar{y}, \bar{w}, \overline{S})$.  Motivated by Propositions \ref{relationSX} and
\ref{prop:polycritical}, in order to state our main result on the metric subregularity of $
\mathcal{T}_l$,
we define the following joint `critical cone' associated with problem   (\ref{dualSDP}) and its constraints as
\begin{equation}\label{defn:criticalkkt}
\mathcal{C}(\bar{y}, \bar{w}, \overline{S}, \overline{X}) :=\left\{
\begin{array}{r}
(d_y, d_w, d_S, d_X)
\\ \in\mathbb{E}
\end{array}
\; \Big|
\begin{array}{l}
d_y\in\mathcal{C}_{\mathcal{Q}}(\bar{y}, b-\mathcal{A}\overline{X}), \;
d_S\in\mathcal{C}_{\mathcal{S}_+^n}(\overline{S}, \overline{X}),\; 
d_X\in\mathcal{C}_{\mathcal{S}_+^n}(\overline{X}, \overline{S}), \; 
\\[5pt]
-\mathcal{A}d_X\in\mathcal{C}_{\mathcal{Q}^\circ}(b-\mathcal{A}\overline{X}, \bar{y}),\; 
\langle d_y, \mathcal{A}d_X\rangle  = 0,
\\[5pt]
\overline{X}^{1/2}d_{{S}}(\overline{S}^{1/2})^\dagger  + (\overline{X}^{1/2})^\dagger d_{{X}} \overline{S}^{1/2} = 0
\end{array}
\right\}.
\end{equation}



\begin{theorem}\label{thm:mainTl}
Let  $(\bar{y}, \bar{w}, \overline{S})$ be an optimal solution to problem  (\ref{dualSDP}) and $\overline{X}\in\mathcal{M}_\psi(\bar{y}, \bar{w}, \overline{S})$.  Let $\bar{q} =  b-\mathcal{A}\overline{X}$ and
$\mathcal{C}_{\mathcal{Q}}(\bar{y}, \bar{q})$ have the form (\ref{eq:defnB}).
Define
$$\left\{\begin{array}{ll}
\mathcal{K}:= (\mathcal{C}_{\mathcal{S}_+^n}(\overline{S}, \overline{X}))^*,\\[5pt]
\Xi:= \left\{(d_y, d_w, d_S, d_X)\in\mathbb{E}\,\mid\, \mathcal{A}^*d_y + \mathcal{F}^*d_w + d_S = 0,\; (\nabla h^*)'(-\bar{w};d_w) + \mathcal{F}d_X = 0\right\}.
\end{array}\right.
$$
Assume that the following three conditions hold:\\[5pt]
(i)
the sets $\mathcal{F}\mathcal{K}$ and
$(\mathcal{A}\; -\mathcal{I})(\mathcal{K},\; \mathcal{C}_{\mathcal{Q}^\circ}(\bar{q},\bar{y})\cap L_a^\perp)$ are closed for all $a\in\mathcal{I}_Q$, where
$\mathcal{I}_Q$ and $L_a$ are defined in (\ref{defn:IB}) and (\ref{defn:LI}); \\[5pt]
(ii) $\langle \Pi_{\mathcal{K}}(d_X), \Pi_{\mathcal{K}}(d_S)\rangle  = 0$ for any  $(d_y,d_w,d_S, d_X)\in\mathcal{C}(\bar{y}, \bar{w}, \overline{S}, \overline{X})\cap \Xi$, where the set $\mathcal{C}(\bar{y}, \bar{w}, \overline{S}, \overline{X})$ is defined in
(\ref{defn:criticalkkt});\\[5pt]
(iii) for problem (\ref{dualSDP}),
 the second order sufficient condition (\ref{sosc:dual}) holds  at $(\bar{y}, \bar{w}, \overline{S})$ with respect to the multiplier $\overline{X}$. \\[5pt]
Then
there exist a constant $\kappa>0$ and a neighborhood $\mathcal{U}$ of $(\bar{y}, \bar{w}, \overline{S},\overline{X})$ such that for any $(u,V)\in\mathbb{E}$,
\begin{equation}\label{dualmetric}
\|(y,w,S) - (\bar{y},\bar{w}, \overline{S})\|\leq \kappa\|(u,V)\|, \quad \forall\; (y,w,S,X)\in\mathcal{T}_l^{-1}(u,V)\cap \mathcal{U}.
\end{equation}
In addition, if $\nabla h^*(\cdot)$ is locally Lipschitz continuous at $-\bar{w}$  and
there exists $\widehat{X}\in\mathcal{M}_\psi(\bar{y}, \bar{w}, \overline{S})$ such that $\textup{rank}\;(\widehat{X}) + \textup{rank}\;(\overline{S})=n$,  then
 $\mathcal{T}_l$ is metrically subregular at $(\bar{y}, \bar{w}, \overline{S},\overline{X})$ for the origin.

%
\end{theorem}

\begin{proof}
 We shall first show that under the given conditions,
  there exist a constant $\kappa>0$ and a  neighborhood $\mathcal{U}$ of $(\bar{z},\overline{X})$ with $\bar{z} := (\bar{y}, \bar{w}, \overline{S})$ such that
 (\ref{dualmetric}) holds. 
 Assume for the sake of contradiction that there exist sequences $\{(y^k,w^k,S^k,X^k)\}\in\mathbb{E}$ and $\{(u_1^k,u_2^k,U^k, V^k)\}\in\mathbb{E}$ with $k\geq 0$ such that $u^k:=(u_1^k,u_2^k,U^k)\to 0$, $V^k\to0$, $z^k:=(y^k,w^k,S^k)\to \bar{z}$, $X^k\to\overline{X}$ with $(z^k, X^k)\in\mathcal{T}_l^{-1}(u^k, V^k)$ and
$$
t_k:=\|z^k - \bar{z}\|\geq \rho_k\|(u^k, V^k)\|, \quad \textup{for some}\; 0<\rho_k \to\infty.
$$
By restricting to an appropriate subsequence if necessary, we may
assume that $(z^k - \bar{z})/{t_k}\to d_{\bar{z}}$ for some $0\neq d_{\bar{z}}:=(d_{\bar{y}}, d_{\bar{w}}, d_{\overline{S}})\in\mathbb{Z}$.
It is easy to see from the KKT optimality condition (\ref{pertkkt}) that for all $k\geq 0$,
\begin{equation}\label{thm:eq1}
0 = \mathcal{A}^*(y^k - \bar{y}) + \mathcal{F}^*(w^k - \bar{w}) + (S^k - \overline{S})  +V^k
\end{equation}
and for all $k$ sufficiently large,
\begin{equation}\label{thm:lim1}
\begin{array}{ll}
0 & = \nabla h^*(-w^k)  - \mathcal{F}X^k+ u_2^k \\[5pt]
& = \nabla h^*(-\bar{w})   - \mathcal{F}\overline{X} +  (\nabla h^*)'(-\bar{w};-w^k + \bar{w}) +r^k - \mathcal{F}(X^k - \overline{X}) + u_2^k  \\[5pt]
& = (\nabla h^*)'(-\bar{w};-w^k + \bar{w})  - \mathcal{F}(X^k - \overline{X})+r^k + u_2^k
\end{array}
\end{equation}
with some $r^k\in\mathbb{W}$ and $r^k = o(t_k)$ as $k\to\infty$.
Dividing both sides of the  equation (\ref{thm:eq1})  by  $t_k$ and then taking limits, we obtain
\begin{equation}\label{critical1}
\mathcal{A}^*d_{\bar{y}} + \mathcal{F}^*d_{\bar{w}} + d_{\overline{S}} = 0.
\end{equation}
For the simplicity, we denote
$$
\Omega := \left\{X\in\mathbb{S}^n\;\mid\;[\overline{P}_{\alpha}\; \overline{P}_{\beta}]^TX\,[\overline{P}_{\alpha}\; \overline{P}_{\beta}] = 0\right\}
$$
and for all $k\geq 0$,
\begin{equation}\label{thm:notation}
\left\{\begin{array}{ll}
{X}_U^k := X^k - U^k, \quad         
\widetilde{X}_U^k := \overline{P}^TX_U^k\overline{P},\quad \widetilde{S}^k := \overline{P}^TS^k\overline{P},
\\[5pt]
 H^k:=\Pi_{{\Omega}}\left(({X_U^k - \overline{X}})/{t_k}\right), \quad G^k := (X_U^k - \overline{X})/{t_k} - H^k\in\mathcal{K}.
\end{array}\right.
\end{equation}
 Using Proposition \ref{relationSX} and $0\in \overline{X} + \partial\delta_{\mathbb{S}_+^n}(\overline{S})$, $0\in X_U^k + \partial\delta_{\mathbb{S}_+^n}(S^k)$ for all $k\geq 0$,  we deduce that for all  $(X^k,S^k)$ sufficiently close to $(\overline{X}, \overline{S})$,
$$\left\{\begin{array}{ll}
\widetilde{S}^k_{\alpha\alpha}  = O(\|S^k - \overline{S}\|\|X_U^k - \overline{X}\|), \quad\widetilde{S}^k_{\alpha\beta} = O(\|S^k - \overline{S}\|\|X_U^k - \overline{X}\|),\\[5pt]
(\widetilde{X}_U^k)_{\beta\gamma}  = O(\|S^k - \overline{S}\|\|X_U^k - \overline{X}\|), \quad(\widetilde{X}_U^k)_{\gamma\gamma} = O(\|S^k - \overline{S}\|\|X_U^k - \overline{X}\|),\\[5pt]
\widetilde{S}^k_{\alpha\gamma} = -\overline{\Lambda}_{\alpha}^{-1}(\widetilde{X}_U^k)_{\alpha\gamma}\overline{\Lambda}_{\gamma} + O(\|S^k - \overline{S}\|\|X_U^k - \overline{X}\|),
\end{array}\right.
$$
which, together with the fact that $\widetilde{S}^k_{\beta\beta}\in\mathbb{S}_+^{|\beta|}$, yields
\begin{equation}\label{thm1:lim11}
d_{\overline{S}}\in\mathcal{C}_{\mathcal{S}_+^n}(\overline{S}, \overline{X}), \quad H_1:=\lim_{k\to\infty} H^k= \overline{P}^T\begin{pmatrix}
 0 & 0 & -\overline{\Lambda}_{\alpha}(\tilde{d}_{\overline{S}})_{\alpha\gamma}\overline{\Lambda}_{\gamma}^{-1}\\
 0 & 0 & 0\\
\left(-\overline{\Lambda}_{\alpha}(\tilde{d}_{\overline{S}})_{\alpha\gamma}\overline{\Lambda}_{\gamma}^{-1}\right)^T & 0 & 0
 \end{pmatrix}\overline{P},
\end{equation}
where $\tilde{d}_{\overline{S}}:= \overline{P}^T d_{\overline{S}}\overline{P}$.
From Proposition \ref{prop:polycritical} and  $b-\mathcal{A}\overline{X}\in\mathcal{N}_{\mathcal{Q}}(\bar{y})$, $u_1^k + b-\mathcal{A}X^k \in\mathcal{N}_{\mathcal{Q}}(y^k)$ for each $k$, we may  assume,  by passing to a subsequence if necessary, that
there exists $a\in\mathcal{I}_Q$ such that for all $k\geq 0$,
$$\mathcal{C}_{\mathcal{Q}}(\bar{y},b-\mathcal{A}\overline{X})\cap L_a \ni (y^k - \bar{y})\perp (u_1^k -\mathcal{A}X^k +\mathcal{A}\overline{X})\in\mathcal{C}_{\mathcal{Q}^\circ}(b-\mathcal{A}\overline{X}, \bar{y})\cap L_a^\perp.
$$
This further implies that
\begin{equation}\label{thm:lim33}
\mathcal{C}_{\mathcal{Q}}(\bar{y},b-\mathcal{A}\overline{X})\cap L_a \ni d_{\bar{y}} \perp  ( (u_1^k-\mathcal{A}U^k)/t_k-\mathcal{A}(H^k + G^k) )\in\mathcal{C}_{\mathcal{Q}^\circ}(b-\mathcal{A}\overline{X}, \bar{y})\cap L_a^\perp.
\end{equation}
In view of  (\ref{thm:lim1}), (\ref{thm:lim33}) and the definitions of $H^k$ and $G^k$ in (\ref{thm:notation}), it follows that for  $k$ sufficiently large,
 $$
\left\{\begin{array}{ll}
 (u_1^k-\mathcal{A}U^k)/t_k-\mathcal{A}H^k  =
\mathcal{A}G^k +\left((u_1^k-\mathcal{A}U^k)/t_k-\mathcal{A}(H^k + G^k) \right)\\[2pt]
\qquad\qquad\qquad\qquad\qquad\;\in (\mathcal{A}\quad \mathcal{I})(\mathcal{K},\; \mathcal{C}_{\mathcal{Q}^\circ}(b-\mathcal{A}\overline{X}, \bar{y}) \cap L_a^\perp),\\[5pt]
(\nabla h^*)'\left(-\bar{w};{-(w^k-\bar{w})}/{t_k}\right) -\mathcal{F}H^k - (\mathcal{F}U^k - r^k - u_2^k)/t_k = \mathcal{F}G^k\in\mathcal{F}\mathcal{K}.
\end{array}\right.
 $$
Since $(\mathcal{A}\quad \mathcal{I})(\mathcal{K},\; \mathcal{C}_{\mathcal{Q}^\circ}(b-\mathcal{A}\overline{X},\bar{y})\cap L_a^\perp)$ and $\mathcal{F}\mathcal{K}$ are assumed to be closed and that (\ref{thm1:lim11}) holds,  there exist
$q\in \mathcal{C}_{\mathcal{Q}^\circ}(b-\mathcal{A}\overline{X}, \bar{y})\cap L_a^\perp$ and ${H}_2\in\mathcal{K}$  such that
 \begin{equation}\label{thm:lim44}
 -\mathcal{A}H_1= \mathcal{A}{H}_2 + q, \quad \quad -(\nabla h^*)'(-\bar{w};d_{\bar{w}})   -\mathcal{F}H_1 = \mathcal{F}{H}_2.
 \end{equation}
Let $d_{\overline{X}} := {H}_1 + {H}_2$. Then  we can obtain from (\ref{critical1}) and (\ref{thm1:lim11})-(\ref{thm:lim44})
that
$(d_{\bar{y}}, d_{\bar{w}}, d_{\overline{S}}, d_{\overline{X}})\in\mathcal{C}(\bar{y}, \bar{w}, \overline{S}, \overline{X})\cap \Xi$.
Furthermore, by using condition (ii) in this theorem,  we have that  $0\neq (d_{\bar{y}},d_{\bar{w}}, d_{\overline{S}})\in\mathcal{C}_{\psi}(\bar{y}, \bar{w}, \overline{S})$ and
$$\begin{array}{rl}
&\langle d_{\bar{w}}, (\nabla h^*)'(-\bar{w};d_{\bar{w}})\rangle +
2\Upsilon_{\overline{S}}(\overline{X}, d_{\overline{S}})\\[2pt]
= & \langle d_{\bar{w}}, -\mathcal{F}d_{\overline{X}}\rangle +
2\langle \overline{X}, d_{\overline{S}}\overline{S}^\dagger d_{\overline{S}}\rangle 
=  \langle \mathcal{A}^*d_{\bar{y}} + d_{\overline{S}}, d_{\overline{X}}\rangle + \langle \overline{\Lambda}_{\alpha}\;,\; (\tilde{d}_{\overline{S}})_{\alpha\gamma}\overline{\Lambda}_{\gamma}^{-1} (\tilde{d}_{\overline{S}})_{\alpha\gamma}^T\rangle
   \\[2pt]
= & \langle  d_{\overline{S}}, d_{\overline{X}}\rangle- 2\langle (\tilde{d}_{\overline{S}})_{\alpha\gamma}\;,\;(\tilde{d}_{\overline{X}})_{\alpha\gamma}\rangle  
=  \langle (\tilde{d}_{\overline{S}})_{\beta\beta}, (\tilde{d}_{\overline{X}})_{\beta\beta}\rangle  = \langle \Pi_{\mathcal{K}}(d_{\overline{S}}), \Pi_{\mathcal{K}}(d_{\overline{X}})\rangle =
0,
\end{array}
$$
which contradicts  the assumed second order sufficient condition (\ref{sosc:dual})
at $(d_{\bar{y}},d_{\bar{w}}, d_{\overline{S}})$ for $\overline{X}$. This  contradiction shows that    there exist a constant $\kappa>0$ and a neighborhood $\mathcal{U}$ of $(\bar{y}, \bar{w}, \overline{S},\overline{X})$ such that
(\ref{dualmetric}) holds.

Next we shall show that if there exists $\widehat{X}\in\mathcal{T}_{\phi}^{-1}(0)$ such that $\textup{rank}\;(\widehat{X}) + \textup{rank}\;(\overline{S})=n$,
 then $\mathcal{T}_l$ is metrically subregular at $(\bar{z}, \overline{X})$ for the origin,   or in view of Definition \ref{defn:metric11}, equivalently to show that
 there exist a constant $\kappa'>0$ and a neighborhood $\mathcal{U}'$ of $(\bar{y}, \bar{w}, \overline{S},\overline{X})$ such that for any $(u,V):=(u_1,u_2,U,V)\in\mathbb{E}$,
\begin{equation}\label{metricTl}
\textup{dist}((y,w,S,X), \mathcal{T}_l^{-1}(0))\leq \kappa' \|(u, V)\|, \; \forall\, (y,w,S,X)\in\mathcal{T}_l^{-1}(u,V)\cap \mathcal{U}'.
\end{equation}
Denote
$$\begin{array}{cc}
\Delta_1 := \{X\in\mathbb{S}^n\mid b-\mathcal{A}X\in\mathcal{N}_{\mathcal{Q}}(\bar{y})\}, \quad\quad \Delta_2: = \{X\in\mathbb{S}^n \mid \mathcal{F}X - \nabla h^*(-\bar{w})=0\},\\ [5pt]
\Delta_3 : = \{X\in\mathbb{S}^n\mid X\in \mathcal{N}_{\mathbb{S}_-^n}(-\overline{S})\}.
\end{array}
$$
Then one has  $\mathcal{T}_\phi^{-1}(0) = \Delta_1\cap\Delta_2\cap\Delta_3$ and $\widehat{X}\in\Delta_1\cap\Delta_2\cap\text{ri}\;(\Delta_3)$.
Thus,
 we obtain from  Proposition \ref{prop:boundedlinear} that there exists a constant $\kappa_1>0$ such that for any $(y,w,S,X)\in\mathcal{U}$,
\begin{equation}\label{Tlineq000}
\textup{dist}(X, \mathcal{T}_\phi^{-1}(0)) \leq \kappa_1\big(\textup{dist}(X, \Delta_1) + \textup{dist}(X, \Delta_2)  + \textup{dist}(X, \Delta_3) \big).
\end{equation}
Consider an arbitrary but fixed point $(y,w,S,X)\in\mathcal{T}_l^{-1}(u,V)\cap \mathcal{U}$.
From Proposition \ref{prop:polyhedral},
 Lemma \ref{lemma:polyhdist} and the fact that $u_1+b-\mathcal{A}X \in\mathcal{N}_{\mathcal{Q}}(y)$, we see that there exist constants $\kappa_2>0$ and $\kappa_2' >0$ such that
\begin{equation}\label{Tlineq111}
\begin{array}{ll}
 \textup{dist}(X,\Delta_1) & \leq \kappa_2\textup{dist}(b-\mathcal{A}X, \mathcal{N}_{\mathcal{Q}}(\bar{y}))\\[2pt]
  &\leq \kappa_2(\|b-\mathcal{A}X - (u_1+b-\mathcal{A}X)\| + \textup{dist}(u_1+b-\mathcal{A}X, \mathcal{N}_{\mathcal{Q}}(\bar{y})))\\[2pt]
  &\leq \kappa_2'(\|u_1\| + \textup{dist}(\bar{y}, \mathcal{N}_{\mathcal{Q}}^{-1}(u_1+b-\mathcal{A}X))) \leq  \kappa_2'(\|u_1\| + \|y-\bar{y}\|),
  \end{array}
  \end{equation}
  where the second inequality comes from the global Lipschitz continuity of $\textup{dist}(\cdot, \mathcal{N}_{\mathcal{Q}}(\bar{y}))$ with modulus $1$.
Using Hoffman's error bound~\cite{hoffman1952approximate} and the assumption on the
local Lipschitz continuity of $\nabla h^*$ at $-\bar{w}$, shrinking  $\mathcal{U}$ if necessary,
 there exist constants $\kappa_3>0$ and $\kappa_3'>0$ such that
\begin{equation}\label{Tline222}
\begin{array}{ll}
 \textup{dist}(X,\Delta_2)& \leq \kappa_3\|\mathcal{F}X - \nabla h^*(-\bar{w})\|\\[2pt]
  &\leq  \kappa_3(\|\mathcal{F}X - \nabla h^*(-{w})\| + \|\nabla h^*(-{w})  - h^*(-\bar{w})\| )\\[2pt]
 &\leq \kappa_3'(\|u_2\| + \|w-\bar{w}\|).
 \end{array}
\end{equation}
Since $\partial\delta_{\mathbb{S}_+^n}(\cdot)= \mathcal{N}_{\mathbb{S}_+^n}(\cdot)$ has been proven to be metrically subregular
 at $\overline{X}$ for $-\overline{S}$ in Proposition \ref{sdp:metricsub}  and $-S\in\mathcal{N}_{\mathbb{S}_+^n}(X-U)$,
  we obtain, by shrinking  $\mathcal{U}$ if necessary,
 that there exists a constant $\kappa_4>0$ such that
\begin{equation}\label{Tlineq333}
\begin{array}{rl}
\textup{dist}(X,\Delta_3) \leq &\|X - (X-U)\| + \textup{dist}(X - U,\mathcal{N}_{\mathbb{S}_-^n}(-\overline{S}))\\[5pt]
 \leq &
\|U\| + \kappa_4\;\textup{dist}(-\overline{S}, \mathcal{N}_{\mathbb{S}_+^n}(X - U))\leq \max\{1,\kappa_4\}(\|U\| + \|S - \overline{S}\|).
\end{array}
\end{equation}
Therefore, combining  the inequality (\ref{dualmetric}) and the
 inequalities (\ref{Tlineq000})-(\ref{Tlineq333}) and, we show that
there exists a constant $\kappa'$ along with a neighborhood $\mathcal{U}'$ such that
(\ref{metricTl}) holds. This completes the proof of this proposition.
\end{proof}
Below, we make some  comments on Theorem \ref{thm:mainTl}.
\begin{remark}\label{remark:Tl}
As one can see, the proof of Theorem \ref{thm:mainTl}
is complicated due to the non-polyhedral nature of the positive semidefinite cone. Here,
we {have} adopted some ideas from the nonlinear programming literature~\cite{dontchev1995characterizations,klatte2000upper,izmailov2013note} on the proof of the metric subregularity of $\mathcal{T}_l$ to our context.
 It is easy to verify via Theorem \ref{thm:mainTl}
that for Example \ref{ex:1}, the operator $\mathcal{T}_l$ is metrically subregular at
any $(\bar{y}, \overline{S}, \overline{X})\in\mathcal{T}_l^{-1}(0)$ with $\overline{X}_{11}>0$ for the origin. The failure of the metric subregularity of $\mathcal{T}_l$ at $(\bar{y}, \overline{S},\overline{X}')$ with $\overline{X}'_{11} =0 $ for the origin is due to the violation of the  second order sufficient condition at $(\bar{y}, \overline{S})$ for $\overline{X}'$.

The assumed condition (i) in Theorem \ref{thm:mainTl}  holds  automatically if $|\beta|=0$ or $|\beta|= 1$, in which 
{case} the set
 $\mathcal{K}$ is a polyhedral cone.  The polyhedral cones and positive semidefinite cones are ``nice cones'' in the terminology of Pataki~\cite[Definition 1]{pataki2007closedness}, where the author also  characterized  the  closedness of these cones under linear transformations~\cite[Theorem 1.1]{pataki2007closedness}.
 It is also clear that if $(\bar{y}, \bar{w}, \overline{S}, \overline{X})$ satisfies the partial strict complementarity, i.e., $|\beta| = 0$, then condition  (ii) in Theorem \ref{thm:mainTl} holds automatically.
 One weaker sufficient condition than this partial strict complementarity
to ensure the validity of
condition (ii)  is that  either
$\Pi_{\mathcal{K}}(d_S)=0$ or $\Pi_{\mathcal{K}}(d_X)=0$ for any $(d_y,d_w,d_S, d_X)\in\mathcal{C}(\bar{y}, \bar{w}, \overline{S}, \overline{X})\cap \Xi$.
\end{remark}

To illustrate the metric subregularity results proved in Theorem \ref{thm:mainTl}, we provide the following example, which is first considered in~\cite{ding2016characterization} for different purposes.

\begin{example}\label{ex:Ex3}
	Consider the following convex quadratic SDP problem:
	\begin{equation}\label{eq:Ex3-d}
		\min \Big\{ \frac{1}{2}(\langle I_2,X\rangle-1)^2 + \delta_{\S_+^2}(X) \mid
		 \langle A,X\rangle \leq  1 
		\Big\},
		\end{equation}
		whose dual (in its equivalent minimization form) can be written as
	\begin{equation}\label{eq:Ex3}
	 \min \Big\{ \delta_{\mathbb{R}_+}(y) - y + \frac{1}{2}w^2 + w  +\delta_{{\S}^2_+}(S) 
 \mid
	 yA +wI_2+ S = 0 \Big\},
	\end{equation}	
			where  $A=\left(\begin{array}{cc}
	1 & -2 \\
	-2 & 1
	\end{array} \right)$.
	Problem (\ref{eq:Ex3}) has a unique solution $(\bar{y},\bar{w},\overline{S})=\left(0,0,0_{2}	\right)$. The critical cone of problem \eqref{eq:Ex3} at $(\bar{y},\bar{w},\overline{S})$ is given by ${\cal C}_{\psi}(\bar{y},\bar{w}, \overline{S})=\{ 0\}$. Thus,  both
conditions (ii) and (iii) imposed in Theorem \ref{thm:mainTl} hold.  	
	Note that the solution set to  problem \eqref{eq:Ex3-d} is given by
\[
	\mathcal{T}_{\phi}^{-1}(0) = \left\{X\in{\S}_+^2 \,\mid\,\langle A,X\rangle\leq  1, \; \langle I_2,X\rangle =1  \right\}.
	\]
	Since $|\beta|\leq 1$ for all $\overline{X}\in\mathcal{T}_{\phi}^{-1}(0)$, we see that condition (i) also holds. Therefore, by Theorem \ref{thm:mainTl} we know that $\mathcal{T}_l$ is metrically subregular at any $(\bar{y}, \bar{w}, \overline{S}, \overline{X})\in\mathcal{T}_l^{-1}(0)$ for the origin, even though the partial strict complementarity condition fails at $\overline{X} = \begin{pmatrix}
1 & \\
 & 0
 \end{pmatrix}
$ or $ \overline{X} =\begin{pmatrix}
0 & \\
 & 1
 \end{pmatrix}$.
\end{example}




\section{Asymptotic superlinear convergence of the ALM with multiple solutions}

In this section, we study the asymptotic superlinear convergence of the ALM for solving problem (\ref{dualSDP}). First, we need to state the PPA considered by Rockafellar~\cite{rockafellar1976monotone}.
Let $\mathcal{T}:\mathbb{X}\rightrightarrows\mathbb{X}$ be a maximal monotone operator. Consider the   following inclusion problem:
$$0\in\mathcal{T}(\xi), \;\forall\, \xi\in\mathbb{X}.
$$
Given a sequence of scalars $c_k\uparrow c_\infty\leq \infty$ and a starting point $\xi^0\in\mathbb{X}$,
the  $(k+1)$-th iteration of the PPA takes the form of
\begin{equation}\label{ppa:iter}
\xi^{k+1} \approx (\mathcal{I} + c_k \mathcal{T})^{-1}(\xi^k), \quad \forall\,k\geq 0.
\end{equation}
For each $k\geq 0$, denote
$$
e^{k+1} := (\mathcal{I} + c_k \mathcal{T})^{-1}(\xi^k)-\xi^{k+1}.
$$
In one of his seminal works~\cite{rockafellar1976monotone}, Rockafellar suggested the following criteria for computing $\xi^{k+1}$ approximately:
\begin{flalign*}
(A)\quad & \|e^{k+1}\|\leq \varepsilon_k, \quad \varepsilon_k \geq 0, \quad   
\mbox{{$\sum_{k=1}^\infty$}}\, \varepsilon_k <\infty, &\\[5pt]
(B) \quad & \|e^{k+1}\| \leq \eta_k\|\xi^{k+1} - \xi^k\|, \quad \eta_k \geq 0, \quad 
\mbox{{$\sum_{k=1}^\infty$}}  \,\eta_k < \infty.
\end{flalign*}

The next  theorem concerning the convergence of the PPA
essentially
comes from Rockafellar~\cite{rockafellar1976monotone} with an important
extension made by Luque~\cite{luque1984asymptotic} on the rate of convergence without assuming the uniqueness of the solutions. For our later developments, here we make a further extension
by relaxing Luque's condition~\cite[(2.1)]{luque1984asymptotic}, which can be too restrictive in our context, in particular when the optimal solution set to problem (\ref{primalSDP})
is unbounded. Note that for the case $e^k\equiv 0$ for all $k\geq 0$, this relaxation has also been discussed by Leventhal~\cite{leventhal2009metric}.

\begin{theorem}\label{thm:ppa}
Assume that $\mathcal{T}^{-1}(0)\neq \emptyset$. Let $\{\xi^k\}$ be an infinite sequence generated by the PPA (\ref{ppa:iter}) with stopping criterion $(A)$. {Then the following statements hold.}
\\[5pt]
(a) For any $\bar{\xi}\in \mathcal{T}^{-1}(0)$, it holds that
\begin{equation}\label{ppa:decrease}
\|\xi^{k+1} + e^{k+1} - \bar{\xi}\|^2  \leq \|\xi^k - \bar{\xi}\|^2 -  \|\xi^{k+1} + e^{k+1} - \xi^k \|^2, \;\forall\, k\geq 0.
\end{equation}
(b) The whole sequence $\{\xi^k\}$ converges to some $\xi^\infty \in \mathcal{T}^{-1}(0)$. Assume that $\mathcal{T}$ is metrically subregular at $\xi^\infty$ for the origin with modulus $\kappa\geq 0$.
If in the PPA, the criterion (B) is also employed, then there exists $\bar{k}\geq 0$ such that for all $k\geq \bar{k}$,  $\eta_k<1$ and
\begin{equation}\label{ppa:rate}
\textup{dist}\,(\xi^{k+1}, \mathcal{T}^{-1}(0))\leq \theta_k \textup{dist}\,(\xi^{k}, \mathcal{T}^{-1}(0)),\end{equation}
where $$
\begin{array}{ll}
1> \theta_k = (\kappa/\sqrt{\kappa^2+ c_k^2} + 2\eta_k)(1-\eta_k)^{-1}\to\theta_\infty = \kappa/\sqrt{\kappa^2+ c_\infty^2} \quad ( \theta_\infty = 0 \; \textup{if}\;  c_\infty = \infty).
\end{array}
$$

\end{theorem}
\begin{proof}
 The inequality (\ref{ppa:decrease}) in part (a) follows directly from~\cite[(2.11)]{rockafellar1976monotone}.  Note that by Definition \ref{defn:metric11}, the metric subregularity of $\mathcal{T}$ at $\xi^{\infty}$ for the origin with modulus $\kappa\geq 0$ is equivalent to the existence of a   neighborhood $\mathcal{U}$ of $\xi^\infty$ such that for all $w\in\mathbb{X}$, 
\begin{equation}\label{ass:eb}
\textup{dist}(\xi, \mathcal{T}^{-1}(0)) \leq \kappa\|w\|, \; \forall\, \xi\in \mathcal{T}^{-1}(w)\cap \mathcal{U}.
\end{equation}
Thus,  to prove (\ref{ppa:rate}) in part (b),  one can
use {a similar
proof as in}~\cite[Theorem 2.1]{luque1984asymptotic} except for replacing  condition (2.1) in~\cite{luque1984asymptotic}
 by   condition (\ref{ass:eb}) with some neighborhood $\mathcal{U}$ of $\xi^\infty$.
 For brevity, we omit the details here.
\end{proof}

%
%
%
%
%
%

Denote $\mathbb{D} := \mathcal{Q}\times \mathbb{W}\times \mathbb{S}_+^n$.  For convenience, we rewrite problems (\ref{primalSDP}) and (\ref{dualSDP}) in the following equivalent forms, respectively:
\begin{equation}\label{defn:sdp:opt}
\begin{array}{cl}
\max &   -h(\mathcal{F}X)  - \langle C,X\rangle\\[5pt]
\text{s.t.} & b-\mathcal{A}X\in\mathcal{Q}^\circ, \; X\in\mathbb{S}_+^n
\end{array}
\end{equation}
and
\begin{equation}\label{dualSDP:alm}
\begin{array}{ll}
\min &  \vartheta(z):=-\langle b,y\rangle +h^*(-w)\\[5pt]
\text{s.t.} & \mathcal{A}^*y + \mathcal{F}^*w + S = C,\; z\in\mathbb{D}.
\end{array}
\end{equation}
Let $\inf\vartheta$ be the optimal value of $\vartheta$ for problem (\ref{dualSDP:alm}). The  Lagrangian function $l$ for problem (\ref{dualSDP:alm})
now takes the  form of
$$l(z,X):=
-\langle b,y\rangle  + h^*(-w)  + \langle X, \mathcal{A}^*y  + \mathcal{F}^*w + S -C\rangle,\;\forall\,(z,X)\in\mathbb{D}\times  \mathbb{S}^n.
$$
The functions $\psi$ and $\phi$ defined in (\ref{defn:phipsi}) can be rewritten as
$$
\psi(z) := \sup_{X\in\mathbb{S}^n} l(z,X),\; \forall\, z\in\mathbb{D},
\quad \quad
\phi(X) := \inf_{z\in\mathbb{D}} l(z,X),\; \forall\, X\in\mathbb{S}^n
$$
while the mappings $\mathcal{T}_l$, $\mathcal{T}_\phi$ and  $\mathcal{T}_\psi$ in (\ref{defn:Tl}) and (\ref{defn:Tphi}) can be reformulated as
$$\begin{array}{cc}
\mathcal{T}_l(z,X):=\big\{(u,v)\in\Z\times \mathbb{S}^n\; {\mid} \; (u,-v)\in\partial l(z,X)\big\}, \;\forall \,(z,X)\in\mathbb{D}\times \S^n,\\[5pt]
\mathcal{T}_\psi(z):=\partial \psi(z), \;\forall\, z\in\mathbb{D},
\quad \quad
\mathcal{T}_\phi(X):=- \partial \phi(X), \;\forall\, X\in\S^n.
\end{array}
$$
Let $c>0$ be a positive parameter.
For any $X\in\mathbb{S}^n$,
the augmented Lagrangian function associated with  problem (\ref{dualSDP:alm}) is given by
$$
L_{c}(z,X):= l(z,X)  + \frac{c}{2}\|\mathcal{A}^*y +\mathcal{F}^*w + S -C\|^2, \;\forall\, z\in\mathbb{D}.
$$
Given a sequence of scalars $c_k\uparrow c_\infty\leq \infty$ and a starting point $X^0\in\mathbb{S}^n_+$,
for $k\ge 0$, the $(k+1)$-th iteration  of the ALM is given by
\begin{equation}\label{iter:alm:main}
\left\{\begin{array}{ll}
z^{k+1} \approx \arg\min\{ \zeta_k(z):=L_{c_k}(z,X^k)\;\mid\; z\in\mathbb{D}\},\\[5pt]
X^{k+1} = X^k  + c_k(\mathcal{A}^*y^{k+1} + \mathcal{F}^*w^{k+1}+S^{k+1}  - C) .
\end{array}\right.\;  
\end{equation}
It is easy to check that if
$
\tilde{z}\in\arg\min \{\zeta_k(z)\;|\;z\in\mathbb{D}\}
$, then we must have $\widetilde{S} = \Pi_{\mathbb{S}_+^n}(C - \mathcal{A}^*\tilde{y} - \mathcal{F}^*\tilde{w} - c_k^{-1}X^k)$.
This motivates us to define for any $k\geq 0$,
\begin{equation}\label{defn:tildeS}
\mathcal{S}_k(y,w): = \Pi_{\mathbb{S}_+^n}(C - \mathcal{A}^*{y} - \mathcal{F}^*{w} - c_k^{-1}X^k), \; \forall\, (y,w)\in\mathcal{Q}\times \mathbb{W}.
\end{equation}
Thus,  for $k\ge 0$,  the $(k+1)$-th iteration  of the ALM in (\ref{iter:alm:main}) can be computed in the following manner
\begin{equation}\label{iter:alm:subS}
\left\{\begin{array}{rl}
(y^{k+1},w^{k+1}) &\approx \arg\min\{\zeta_k(y,w, \mathcal{S}_k(y,w)) \mid y\in\mathcal{Q},\;w\in\mathbb{W}\},\\[2pt]
S^{k+1} & = \mathcal{S}_k(y^{k+1}, w^{k+1}),\\[2pt]
X^{k+1} &= X^k  + c_k(\mathcal{A}^*y^{k+1} + \mathcal{F}^*w^{k+1}+S^{k+1}  - C) \\[2pt]
& = \Pi_{\mathbb{S}_+^n}(X^k  + c_k(\mathcal{A}^*y^{k+1} + \mathcal{F}^*w^{k+1}- C)).
\end{array}\right. 
\end{equation}
In accordance with Rockafellar's work in~\cite{rockafellar1976augmented}, we shall terminate the subproblem for solving $z^{k+1}$ in (\ref{iter:alm:main}) by the following three criteria:
\begin{flalign*}
(A')\quad  &\zeta_k(z^{k+1}) - \inf\zeta_k \leq \varepsilon_k^2/2c_k, \quad \varepsilon_k\geq 0, \quad \mbox{$\sum_{k=0}^\infty$}\, \varepsilon_k <  \infty, &\\[5pt]
(B') \quad & \zeta_k(z^{k+1}) - \inf\zeta_k \leq (\eta_k^2/2c_k)\|X^{k+1}  - X^k\|^2, \quad  \eta_k\geq 0, \quad \mbox{$\sum_{k=0}^\infty$}\,  \eta_k <  \infty, &\\[5pt]
(B'')\quad & \textup{dist}\,(0,\partial \zeta_k(z^{k+1}))\leq (\eta_k'/c_k)\|X^{k+1} - X^k\|, \quad 0\leq \eta_k' \to 0.&
\end{flalign*}

 A notable result of Rockafellar~\cite{rockafellar1976augmented} shows that
the ALM in (\ref{iter:alm:main}) for solving the dual problem (\ref{dualSDP}) with criteria $(A')$ and $(B')$ can be viewed as the PPA
applied to $\mathcal{T}_\phi = \partial \phi$ as {in} (\ref{ppa:iter}) with {stopping}
criteria $(A)$ and $(B)$. This will help us  to
obtain the  global convergence and the asymptotic superlinear convergence rates of the ALM for solving problem (\ref{dualSDP}). But first, we need the following simple property.

\begin{prop}\label{prop:alm:ineq}
Let $\{(z^k,X^k)\}$ be a sequence generated by the ALM (\ref{iter:alm:main}) under criterion $(B')$. Then for all $k\geq 0$ such that $\eta_k<1$, it holds that
$$\|X^{k+1} - X^k\|\leq (1-\eta_k)^{-1}\textup{dist}(X^k,\mathcal{T}_{\phi}^{-1}(0)).
$$
\end{prop}
\begin{proof}
By using Theorem \ref{thm:ppa} (a) and criterion $(B')$,
 we get for all $k\geq 0$ that
 $$\begin{array}{ll}
 \|X^{k+1} - X^k\|& \leq \|X^{k+1} +e^{k+1} - X^k\| + \|e^{k+1}\|\\[2pt]
 & \leq \textup{dist}(X^k, \mathcal{T}_{\phi}^{-1}(0)) + (2c_k(\zeta_k(z^{k+1}) - \inf \zeta_k))^{1/2}\\[2pt]
 & \leq \textup{dist}(X^k, \mathcal{T}_{\phi}^{-1}(0)) + \eta_k\|X^{k+1} - X^k\|,
 \end{array}
 $$
 where the second term in the second inequality comes from~\cite[Proposition 6]{rockafellar1976augmented}. 
Thus, the conclusion of this proposition holds.
\end{proof}
\begin{theorem}\label{thm:alm}
Assume that the optimal solution set $\mathcal{T}_{\phi}^{-1}(0)$ to problem (\ref{primalSDP}) is non-empty.
 Let $\{(z^k,X^k)\}$ be an infinite  sequence generated by the ALM in (\ref{iter:alm:subS})  with stopping criterion $(A')$.
Then, the whole sequence $\{X^k\}$ is bounded and converges to some $X^\infty \in \mathcal{T}_\phi^{-1}(0)$, and  the sequence $\{z^k\}$ satisfies for all $k\geq 0$,
 $z^k\in\mathbb{D}$ and
\begin{equation}\label{alm:feas:k}
\|\mathcal{A}^*y^{k+1} + \mathcal{F}^*w^{k+1} + S^{k+1}-C\|= c_k^{-1}\|X^{k+1} - X^k\|\to 0,
\end{equation}
\begin{equation}\label{alm:obj:k}
\vartheta(z^{k+1})-\inf \vartheta \leq \zeta_k(z^{k+1}) - \inf \zeta_k + (1/2c_k)(\|X^k\|^2 - \|X^{k+1}\|^2).
\end{equation}
Moreover, if  problem (\ref{dualSDP}) admits a non-empty and bounded solution set, then the sequence $\{z^k\}$ is also bounded, and all of its accumulation points are optimal solutions to problem (\ref{dualSDP}).

Regarding  the convergence rates of the ALM, we have the following {results.}\\[5pt]
(a)
If  $\mathcal{T}_\phi$ is metrically subregular at $X^\infty$ for the origin with modulus $\kappa_\phi>0$, then   under criterion $(B')$: there exists $\bar{k} \geq 0$ such that for all $k\geq \bar{k}$, $\eta_k<1$ and
\begin{equation}\label{alm:rate}
\textup{dist}\,(X^{k+1}, \mathcal{T}_{\phi}^{-1}(0))\leq \theta_k \,\textup{dist}\,(X^{k}, \mathcal{T}_{\phi}^{-1}(0)),
\end{equation}
\begin{equation}\label{alm:feas}
\|\mathcal{A}^*y^{k+1} + \mathcal{F}^*w^{k+1} + S^{k+1}-C\|\leq \tau_k\,\textup{dist}(X^k, \mathcal{T}_{\phi}^{-1}(0)),\end{equation}
\begin{equation}\label{alm:obj}
\vartheta(z^{k+1})-\inf \vartheta \leq \tau_k'\,\textup{dist}(X^k, \mathcal{T}_{\phi}^{-1}(0)),\end{equation}
where
$$
\begin{array}{ll}
1> \theta_k = \left(\kappa_\phi/\sqrt{\kappa_\phi^2+ c_k^2}  + 2\eta_k\right)(1-\eta_k)^{-1}\to\theta_\infty = \kappa_\phi/\sqrt{\kappa_\phi^2+ c_\infty^2} 
\quad ( \theta_\infty = 0 \; \textup{if}\;  c_\infty = \infty),\\[5pt]
\tau_k=c_k^{-1}(1-\eta_k)^{-1}\to\tau_{\infty} = 1/c_{\infty} \; ( \tau_\infty = 0 \; \textup{if}\;  c_\infty = \infty),\\[5pt]
\tau_k' = \tau_k\left(\eta_k^2\|X^{k+1} - X^{k}\| + \|X^{k+1}\| + \|X^{k}\|\right)/2\to\tau'_{\infty} = \|X^\infty\|/c_{\infty} \; ( \tau'_\infty = 0 \; \textup{if}\;  c_\infty = \infty).\\[5pt]
\end{array}
$$
%
(b) If in addition to $(B')$ and the metric subregularity of $\mathcal{T}_\phi$ at $X^\infty$ for the origin, one has
$(B'')$, $\mathcal{T}_\psi^{-1}(0)$ is non-empty and bounded and the following condition on $\mathcal{T}_l$:   there exist two constants $\kappa_l\geq 0$ and $\epsilon >0$ such that for any $(z,X)\in \mathbb{Z}\times \mathbb{S}^n$ satisfying $\textup{dist}((z,X), \mathcal{T}_{\psi}^{-1}(0)\times \{X^\infty\})\leq \epsilon$,
\begin{equation}\label{alm:assump}
\begin{array}{ll}
\textup{dist}((z,X),\mathcal{T}_l^{-1}(0))\leq \kappa_l\,\textup{dist}(0,\mathcal{T}_l(z,X)).
\end{array}
\end{equation}
Then there exists $\tilde{k} \geq 0$ such that  for all $k\geq \tilde{k}$, $\eta_k <1$ and
\begin{equation}\label{thm:rate:dual}
\textup{dist}(z^{k+1}, \mathcal{T}_\psi^{-1}(0))\leq \theta_k'\textup{dist}(X^k, \mathcal{T}_{\phi}^{-1}(0)),
\end{equation}
where $$\theta_k'= \kappa_lc_k^{-1}(1+\eta_k')(1-\eta_k')^{-1}\to\theta_\infty' = \kappa_l/c_\infty  \; ( \theta'_\infty = 0 \; \textup{if}\;  c_\infty = \infty).$$
\end{theorem}

\begin{proof}
The convergence on the sequences $\{X^k\}$ and $\{z^k\}$
follows from~\cite[Theorem 4]{rockafellar1976augmented}. The inequality (\ref{alm:feas:k}) can be obtained by the definition of $X^{k+1}$ in (\ref{iter:alm:subS}) and the convergence of $\{X^k\}$. By noticing of
$ \vartheta(z^{k+1}) -  \zeta_k(z^{k+1}) = (1/2c_k)(\|X^{k}\|^2 - \|X^{k+1}\|^2)$ and $\inf\zeta_k\leq \inf\vartheta$~\cite[(4.16)--(4.17)]{rockafellar1976augmented}, we get (\ref{alm:obj:k}).
Next, we prove the results on the convergence rates.\\[5pt]
(a) The inequality (\ref{alm:rate}) follows from Theorem \ref{thm:ppa} (b) directly. By combining (\ref{alm:obj:k}), (\ref{alm:feas:k}),
 (\ref{alm:rate}), Proposition \ref{prop:alm:ineq} and criterion $(B')$, we can obtain the inequalities (\ref{alm:feas}) and (\ref{alm:obj}).\\[5pt]
(b) Since $\mathcal{T}_{\psi}^{-1}(0)$ is  assumed to be non-empty and bounded,  we know that the sequence $\{z^k\}$ is bounded and $\textup{dist}(z^k, \mathcal{T}_{\psi}^{-1}(0))\to 0$.
{Thus}
there exists $\tilde{k}\geq 0$ such that for all $k\geq \tilde{k}$, $\eta_k < 1$ and
  $\textup{dist}((z^{k+1},X^{k+1}), \mathcal{T}_{\psi}^{-1}(0)\times \{X^\infty\})\leq \epsilon$. By using condition (\ref{alm:assump}), we have for all $k\geq \tilde{k}$,
 $$
\begin{array}{rl}
\textup{dist}(z^{k+1}, \mathcal{T}_{\psi}^{-1}(0))
\leq 
& \textup{dist}((z^{k+1},X^{k+1}), \mathcal{T}_{l}^{-1}(0))
\\[3pt]
\leq   & \kappa_l\,\textup{dist}(0,\mathcal{T}_l(z^{k+1}, X^{k+1}))
\\[3pt]
\leq  & \kappa_l\,(\textup{dist}^2(0,\partial \zeta_k(z^{k+1})) + c_k^{-2}\|X^{k+1} - X^k\|^2)^{1/2}\\[3pt]
\leq  & \kappa_l(1+\eta_k')c_k^{-1}\|X^{k+1} - X^k\|\\[3pt]
\leq & \kappa_l(1+\eta_k')c^{-1}_k(1-\eta_k)^{-1}\textup{dist}(X^k,\mathcal{T}_{\phi}^{-1}(0)),
\end{array}
$$
where the third inequality comes from~\cite[(4.21)]{rockafellar1976augmented}, the forth inequality 
is due to criterion $(B'')$ and the last inequality follows from Proposition \ref{prop:alm:ineq}.
Thus for $k\geq \tilde{k}$,
the inequality (\ref{thm:rate:dual}) holds.
\end{proof}
\begin{remark}
In Theorem \ref{thm:alm},  under the metric subregularity of $\mathcal{T}_{\phi}$  at $X^\infty$ for the origin,
 the  sequence $\{X^k\}$ is proved to converge Q-(super)linearly to the optimal solution set $\mathcal{T}_{\phi}^{-1}(0)$ to problem (\ref{primalSDP}), while the feasibility and the objective function value of  problem (\ref{dualSDP:alm})  converge at least R-(super)linearly. For the  asymptotic R-superlinear convergence of the iteration sequence $\{z^k\}$ itself, one has to impose a stronger condition on $\mathcal{T}_l$  as in  part (b). In numerical computations one does not need $c_k$ to converge to $+\infty$, instead one can just progressively choose $c_k$ to be large enough, such as $c_k\approx \kappa_{\phi}$, to achieve fast linear convergence.
{Of course one does not know $\kappa_\phi$ in practice, and  hence the adjustment of $c_k$ to achieve
fast linear convergence is always an important issue in the practical implementation of the ALM.}
The metric subregularity of $\mathcal{T}_{\phi}$ at $X^\infty$ for the origin is satisfied in one of the two situations in Corollary \ref{coro:sdpTphi}.
Another situation for ensuring $\mathcal{T}_{\phi}$ to be metrically subregular at $X^\infty$ for the origin
is when  the function $h$  is twice continuously differentiable
and the ``no-gap'' second order sufficient condition holds at ${X}^\infty$~\cite[Theorem 3.137]{bonnans2013perturbation}. Thus, we can see that the metric subregularity of $\mathcal{T}_{\phi}$ at $X^\infty$ for the origin is quite  mild. However,
the metric subregularity of $\mathcal{T}_{l}$ can be more restrictive (refer to Remark \ref{remark:Tl}).

\end{remark}

\subsection{On the implementable stopping criteria for solving the ALM subproblems}

In this subsection, we shall study the implementation issues for applying the ALM to solve problem (\ref{dualSDP:alm}).
While it is relatively easy to  implement
criterion $(B'')$~\cite[(4.6)]{rockafellar1976augmented},  it can be a challenging task to execute
 criteria $(A')$ and $(B')$ since the value $\inf\zeta_k$ is {not} available.
In the following, we shall take the least squares SDP problem with equality constraints, i.e., $h(w) = \frac{1}{2}\|w-d\|^2$ for any $w\in\mathbb{W}$ with given $d\in\mathbb{W}$ and $\mathcal{Q} = \mathbb{Y}$ in problem (\ref{defn:sdp:opt}),
 as an example to illustrate how to implement criteria $(A')$ and $(B')$.
Denote $\mathcal{X}:=\{X\in\mathbb{S}_+^n\,\mid\, \mathcal{A}X - b=0\}$ as the feasible set to problem (\ref{defn:sdp:opt}) in this case.
Here,  we always assume that there exists a strictly feasible point  $\widehat{X}\in\mathcal{X}$ such that
\begin{equation}\label{defn:strictfeas}
\mathcal{A}\widehat{X} - b=0, \quad \quad \widehat{X}\succ 0.
\end{equation}
Denote  $\sigma_{\textup{min}}(\mathcal{A})$ as the smallest  positive singular value of $\mathcal{A}$ and define
\begin{equation}\label{defn:barmu}
\bar{\mu}: = \sigma_{\textup{min}}^{-1}(\mathcal{A})\max\left\{ \lambda_{\textup{min}}^{-1}(\widehat{X}), \; 1+  \lambda^{-1}_{\textup{min}}(\widehat{X})\|\widehat{X}\|\right\}.
\end{equation}
 The following proposition provides an upper bound for the
 distance of  an $X\in\mathbb{S}_+^n$ to the set $\mathcal{X}$.
 \begin{prop}\label{prop:feasalm}
Let $X\in\mathbb{S}_+^n$ be given. Then
$$
\|X - \Pi_{\mathcal{X}}(X)\|\leq  \sigma_{\textup{min}}^{-1}(\mathcal{A})\big(1+\lambda_{\textup{min}}^{-1}(\widehat{X})\|X - \widehat{X}\| \big)\|b-\mathcal{A}X\| \leq
 \bar{\mu}(1+\|X\|)\|b-\mathcal{A}X\|.
$$
\end{prop}
\begin{proof}
Denote $$
 u: = \mathcal{A}X - b, \quad
\Delta X:=-\mathcal{A}^*(\mathcal{A}\mathcal{A}^*)^{\dagger} u, \quad\widecheck{X} := X + \Delta X.$$
Since $b,u\in\textup{Range}\,(\mathcal{A})$, it holds that
{$$
\mathcal{A}\widecheck{X} = \mathcal{A}X + \mathcal{A}\Delta X = b + u-\mathcal{A}\mathcal{A}^*(\mathcal{A}\mathcal{A}^*)^{\dagger} u = b,
\quad \quad
\|\Delta X\|\leq \sigma^{-1}_{\textup{min}}(\mathcal{A})\|u\|.
$$}
Define
$$
{\tau}:=
  \frac{\|u\|}{\|u\| +\sigma_{\textup{min}}(\mathcal{A}) \lambda_{\textup{min}}(\widehat{X})}, \quad  \quad
X': = (1-{\tau})\widecheck{X} + {\tau} \widehat{X}.
$$
Obviously $\tau\in [0,1]$ and $X'\in\mathcal{X}$.
Therefore, we obtain that
$$
\begin{array}{ll}
\|X - \Pi_{\mathcal{X}}(X)\|  &\leq  \|X - X'\|
\leq \|\Delta X\| + \tau\|X-\widehat{X}\| 
 \; \leq \; \sigma_{\textup{min}}^{-1}(\mathcal{A})(1+（\lambda_{\textup{min}}^{-1}(\widehat{X})\|X-\widehat{X}\|)\|u\| \\[5pt]
& \leq \sigma_{\textup{min}}^{-1}(\mathcal{A})(1+\lambda_{\textup{min}}^{-1}(\widehat{X})\|\widehat{X}\| + \lambda_{\textup{min}}^{-1}(\widehat{X})\|{X}\|)\|u\| \\[5pt]
& \leq \sigma_{\textup{min}}^{-1}(\mathcal{A})\max\left\{ \lambda_{\textup{min}}^{-1}(\widehat{X}), \; 1+  \lambda^{-1}_{\textup{min}}(\widehat{X})\|\widehat{X}\|\right\}(1+\|X\|)\|u\|,
\end{array}
$$
which completes the proof of this proposition.
\end{proof}
For any $k\geq 0$,
denote $f_k(X):= -h(\mathcal{F}X) - \langle C,X\rangle - \|X - X^k\|^2/2c_k$ for $X\in\mathbb{S}^n$.
\begin{prop}\label{prop:criteria}
Assume that $\mathcal{A}:\mathbb{S}^n\to\mathbb{Y}$ is surjective and that condition (\ref{defn:strictfeas}) is satisfied. Let  $\bar{\mu}$ be given by (\ref{defn:barmu}) and $\bar{\nu}$ be any positive constant.
Suppose that for some $k\geq 0$, $\varepsilon_k>0$, $\eta_k>0$ and $X^k\in\mathbb{S}_+^n$ is not an optimal solution to  problem  (\ref{defn:sdp:opt}).
Let $\{z^{k,j}\}_{j\geq 0}$ be any sequence such that
$\zeta_k(z^{k,j})\to\inf \zeta_k$
with $(y^{k,j},w^{k,j})\in \mathbb{Y}\times \mathbb{W}$ and $S^{k,j} = \mathcal{S}_k(y^{k,j}, w^{k,j})$, where $\mathcal{S}_k(\cdot)$ is defined as in (\ref{defn:tildeS}).
For any $j\geq 0$, let $$
\begin{array}{cc}
X^{k,j} := \Pi_{\mathbb{S}_+^n}(X^k + c_k(\mathcal{A}^*y^{k,j} + \mathcal{F}^*w^{k,j}-C)),\\[3pt]
u^{j} := \mathcal{A}X^{k,j} - b,\quad
t_{j}: =\bar{\nu}^{-1}\min\left\{\varepsilon_k^2/2c_k,\; (\eta_k^2/2c_k)\|X^{k,j}  - X^k\|^2\right\}.
\end{array}$$
Then there exists $\bar{j}\geq 0$ such that for any $j\geq \bar{j}$,
\begin{equation}\label{prop:alm:criteria:ineq}
\zeta_k(z^{k,j}) - f_k(X^{k,j})\leq t_j,\quad
(1+\|X^{k,j}\|)\|u^{k,j}\|\leq \min\left\{1,\sqrt{c_k}, \sqrt{t_j}/\|\nabla f_k(X^{k,j})\|\right\}\sqrt{t_j}\, .
\end{equation}
Consequently, for all $j\geq \bar{j}$,
$$
\zeta_k(z^{k,j}) - \inf \zeta_k \leq \left(1+ \bar{\mu} + \frac{1}{2}\lambda_{\max}(\mathcal{F}^*\mathcal{F})\bar{\mu}^2 +\frac{1}{2}\bar{\mu}^2\right)t_j.
$$

\end{prop}

\begin{proof}
Since $f_k(\cdot)$ is strongly concave, $\zeta_k(z^{k,j})\to\inf \zeta_k$ and condition (\ref{defn:strictfeas}) is satisfied with $\mathcal{A}$ being surjective, we know from~\cite[Theorems 17 $\&$ 18]{rockafellar1974conjugate} that the  two sequences $\{z^{k,j}\}$ and
$\{X^{k,j}\}$ are bounded, and
 $\{X^{k,j}\}$ converges to some point $X^{k,\infty}\in\mathcal{X}$ such that $f_k(X^{k,\infty}) = \inf \zeta_k$.
One can easily prove that
$X^{k,\infty
	}\neq X^k$, because otherwise $X^{k}$ and any accumulation point of $\{z^{k,j}\}$ forms a KKT solution point to problems  (\ref{defn:sdp:opt}) and (\ref{dualSDP:alm}), which would contradict our assumption that $X^k$ is not an optimal solution to problem  (\ref{defn:sdp:opt}). Thus, for all $j$ sufficiently large, $t_j$ is bounded away from $0$. Then,  there exists $\bar{j}\geq 0$ such that for all $j\geq \bar{j}$,  the two inequalities in (\ref{prop:alm:criteria:ineq}) hold.

By using Proposition \ref{prop:feasalm} and (\ref{prop:alm:criteria:ineq}),
 we  get for all $j\geq \bar{j}$ that
$$\begin{array}{rl}
 &\zeta_k(z^{k,j}) - \inf\zeta_k \\[5pt]
 \leq &  \zeta_k(z^{k,j}) - f_k(\Pi_{\mathcal{X}}(X^{k,j}))
\;=\;   (\zeta_k(z^{k,j}) - f_k(X^{k,j})) + (f_k(X^{k,j}) - f_k(\Pi_{\mathcal{X}}(X^{k,j})))\\[5pt]
\leq & t_j  - \langle \nabla f_k(X^{k,j}), \Pi_{\mathcal{X}}(X^{k,j}) - X^{k,j}\rangle + \frac{1}{2}\langle \Pi_{\mathcal{X}}(X^{k,j}) - X^{k,j}, (\mathcal{F}^*\mathcal{F} + c_k^{-1}\mathcal{I})(\Pi_{\mathcal{X}}(X^{k,j}) - X^{k,j})\rangle\\[5pt]
\leq & t_j + \bar{\mu}\|\nabla f_k(X^{k,j})\|(1+\|X^{k,j}\|)\|u^{k,j}\|  + \frac{1}{2}(\lambda_{\max}(\mathcal{F}^*\mathcal{F})  + c_k^{-1})\bar{\mu}^2(1+\|X^{k,j}\|)^2\|u^{k,j}\|^2\\[5pt]
 \leq & t_j + \bar{\mu} t_j + \frac{1}{2}\lambda_{\max}(\mathcal{F}^*\mathcal{F})\bar{\mu}^2t_j + \frac{1}{2}\bar{\mu}^2t_j.
\end{array}
$$
This completes the proof of the proposition.
\end{proof}

%
%
Proposition \ref{prop:criteria} says that if  $\mathcal{A}$ is surjective and condition (\ref{defn:strictfeas}) 
{is} satisfied,
 we can use the following criteria  to replace $(A')$ and $(B')$ with  some $\bar{\nu}>0$:
 \begin{flalign*}
(\widetilde{A}')\quad &
\left\{\begin{array}{ll}
 \zeta_k(z^{k+1}) - f_k(X^{k+1}) \leq  \bar{\nu}^{-1}\varepsilon_k^2/2c_k, \\[2pt]
(1+\|X^{k+1}\|)\|u^{k+1}\|\leq \min\left\{1,\sqrt{c_k}, \sqrt{t_{k,1}}/\|\nabla f_k(X^{k+1})\|\right\}\sqrt{t_{k,1}}\,,
\end{array}\right.
\varepsilon_k\geq 0, \quad \sum_{k=0}^\infty \varepsilon_k <  \infty, &\\[5pt]
(\widetilde{B}')\quad &
\left\{\begin{array}{ll}
 \zeta_k(z^{k+1}) - f_k(X^{k+1}) \leq \bar{\nu}^{-1}(\eta_k^2/2c_k)\|X^{k+1}  - X^k\|^2, \\[2pt]
(1+\|X^{k+1}\|)\|u^{k+1}\|\leq \min\left\{1,\sqrt{c_k}, \sqrt{t_{k,2}}/\|\nabla f_k(X^{k+1})\|\right\}\sqrt{t_{k,2}}\,,
\end{array}\right.
\eta_k\geq 0, \quad \sum_{k=0}^\infty \eta_k <  \infty, &
\end{flalign*}
where $t_{k,1}:=  \bar{\nu}^{-1}\varepsilon_k^2/2c_k$, $t_{k,2}:=  \bar{\nu}^{-1}(\eta_k^2/2c_k)\|X^{k+1}  - X^k\|^2$ and $u^{k+1}: = \mathcal{A}X^{k+1} - b$.
By taking $\bar{\nu}=1+ \bar{\mu} + \frac{1}{2}\lambda_{\max}(\mathcal{F}^*\mathcal{F})\bar{\mu}^2 +\frac{1}{2}\bar{\mu}^2$, we see that criteria $(A')$ and $(B')$ are satisfied as long as both
$(\widetilde{A}')$ and $(\widetilde{B}')$ are true.  Actually, by taking $\bar{\nu}$ to be any positive constant satisfying $(\widetilde{A}')$ and $(\widetilde{B}')$, we  obtain a sequence $\{(z^k, X^k)\}$ that achieves $(A')$ and $(B')$ with $\{\varepsilon_k\}$ and $\{\eta_k\}$ being replaced by $\{\sqrt{\bar{\nu}}\varepsilon_k\}$ and $\{\sqrt{\bar{\nu}}\eta_k\}$, respectively. All the results in Theorem \ref{thm:alm} are valid with these two new sequences $\{\sqrt{\bar{\nu}}\varepsilon_k\}$ and $\{\sqrt{\bar{\nu}}\eta_k\}$.   As far as we know, these implementable criteria of the ALM {are new.}


Note that while the  assumption on the existence of $\widehat{X}$ in (\ref{defn:strictfeas}) is crucial to our analysis in Proposition \ref{prop:criteria}, the assumption on the surjectivity of the linear operator $\mathcal{A}:\mathbb{S}^n\to\mathbb{Y}$
 is not essential  as we can always redefine $\mathbb{Y} = \textup{Range}\,(\mathcal{A})$, to make $\mathcal{A}$ to be surjective from $\mathbb{S}^n$ to $\mathbb{Y}$.  Additionally, it is not difficult to extend
 the results in Proposition \ref{prop:criteria} to the case that  the constraint set $\mathcal{X}$ in (\ref{defn:sdp:opt}) is replaced by $\{X\in\mathbb{S}_+^n\mid \mathcal{A}_1X = b_1,\;b_2-\mathcal{A}_2X\in\mathcal{Q}'\}$ for some convex polyhedral cone $\mathcal{Q}'$,  if one assumes the existence of $\widehat{X}\succ 0$ such that
$\mathcal{A}_1\widehat{X} = b_1$ and $b_2 - \mathcal{A}_2\widehat{X}\in\textup{int}\,(\mathcal{Q}')$.

\section{Concluding discussions}
In this paper, we have 
{established}
asymptotic superlinear convergence results of the ALM  for solving SDP problems with multiple solutions.
These results can be used to explain the success of the solvers SDPNAL, SDPNAL+  and  QSDPNAL for solving linear and convex quadratic SDP problems.
There are several important issues that are worth further investigations.
 For example,
 it would be interesting to study  whether
our results can be extended to other optimization problems with non-polyhedral constraints
 besides SDP problems. Another important  line of research is to further characterize
the  metric subregularity of $\mathcal{T}_{\phi}$ other than the ones stated in Corollary \ref{coro:sdpTphi} or the ``no-gap'' second order sufficient condition.  We also believe that
it is  worth the effort to investigate the metric subregularity of $\mathcal{T}_l$ under weaker conditions than the ones given  in Theorem \ref{thm:mainTl}. Last but not least, one may ask how much the
obtained results can help {in finding ways to} improve the efficiency of the existing solvers for solving SDP problems,
or even better, to obtain new and more efficient solvers for solving general large scale convex optimization problems.
\bigskip

\noindent{\bf Acknowledgements}
The authors would like to thank Drs. Chao Ding and  Xudong Li for  many helpful discussions on the metric subregularity of the Karush-Kuhn-Tucker solution mappings for both nonlinear programming and semidefinite programming problems.

{\small

\bibliographystyle{plain}
}

\end{document}